\documentclass[11pt]{amsart}
\usepackage[leqno]{amsmath}
\usepackage{amsmath,amsfonts,amsthm,amssymb}
\usepackage{epsfig}
\usepackage{graphicx}
\usepackage{subfigure}
\usepackage{cite,url}
\usepackage{hyperref}
\usepackage{epsfig}
\usepackage{color}
\usepackage{amsmath}
\usepackage{amssymb}
\newtheorem{lemma}{Lemma}
\newtheorem{proposition}{Proposition}
\newtheorem{theorem}{Theorem}
\newtheorem{corollary}{Corollary}

\theoremstyle{definition}
\newtheorem{definition}{Definition}
\newtheorem{question}{Question}

\theoremstyle{remark}
\newtheorem{conjecture}{Conjecture}
\newtheorem{prob}{Problem}

\newcommand{\covectors}{\ensuremath{\mathcal{L}}}

\newcommand{\Min}{{\rm Min}}

\newcommand{\uparr}{\ensuremath{\hspace{3pt}\uparrow\hspace{-2pt}}}

\begin{document}

\thispagestyle{empty}

\centerline{\Large\bf COMs: Complexes of Oriented Matroids}

\vspace{10mm}

\centerline{Hans-J\"{u}rgen Bandelt$^{\small 1}$, Victor Chepoi$^{\small 2}$, and Kolja Knauer$^{\small 2}$}

\medskip
\begin{small}

\centerline{$^{1}$Fachbereich Mathematik, Universit\"at Hamburg,}
\centerline{Bundesstr. 55, 20146 Hamburg, Germany,}

\centerline{\texttt{bandelt@math.uni-hamburg.de}}

\medskip
\centerline{$^{2}$Laboratoire d'Informatique Fondamentale, Aix-Marseille Universit\'e and CNRS,}
\centerline{Facult\'e des Sciences de Luminy, F-13288 Marseille Cedex 9, France}

\centerline{\texttt{\{victor.chepoi, kolja.knauer\}@lif.univ-mrs.fr}}

\end{small}

 \bigskip\bigskip\noindent
{\footnotesize {\bf Abstract.}  In his seminal 1983 paper, Jim Lawrence introduced lopsided sets and featured them as asymmetric
counterparts of oriented matroids, both sharing the key property of strong elimination. Moreover, symmetry of faces holds in both
structures as well as in the so-called affine oriented matroids. These two fundamental properties (formulated for covectors)
together lead to the natural notion of ``conditional oriented matroid'' (abbreviated COM). These novel structures can be
characterized in terms of three cocircuits axioms, generalizing the familiar characterization for oriented matroids. We
describe a binary composition scheme by which every COM can successively be erected as a certain complex of oriented matroids,
in essentially the same way as a lopsided set can be glued together from its maximal hypercube faces. A realizable COM is
represented by a hyperplane arrangement restricted to an open convex set. Among these are the examples formed by linear
extensions of ordered sets, generalizing the oriented matroids corresponding to the permutohedra. Relaxing realizability
to local realizability, we capture a wider class of combinatorial objects: we show that non-positively curved Coxeter
zonotopal complexes give rise to locally realizable COMs.

}

\tableofcontents

\section{Introduction}

\subsection{Avant-propos}
Co-invented by Bland $\&$ Las Vergnas~\cite{bl-lv-78} and Folkman $\&$ Lawrence~\cite{fo-la-78}, and further investigated  by Edmonds $\&$ Mandel~\cite{ed-ma-82} and
many other authors, oriented matroids  represent a unified combinatorial theory of orientations of ordinary matroids, which simultaneously captures the basic
properties of sign vectors representing the regions in a hyperplane arrangement  in ${\mathbb R}^n$ and of  sign vectors of the circuits  in a directed
graph. Furthermore, oriented matroids find applications in point and vector configurations, convex polytopes, and linear programming. Just as ordinary matroids,
oriented matroids may be defined in a multitude of distinct but  equivalent ways: in terms of cocircuits, covectors,
topes, duality, basis orientations, face lattices, and arrangements of pseudospheres.  A full account of the theory of oriented matroids is provided  in the
book by Bj\"orner, Las Vergnas, White, and Ziegler~\cite{bjvestwhzi-93} and an introduction to this rich theory is given in the textbook by Ziegler~\cite{Zie-07}.

Lopsided sets of sign vectors  defined by Lawrence~\cite{la-83} in order to capture the intersection patterns of convex sets with the orthants  of ${\mathbb R}^d$
(and further investigated in~\cite{bachdrko-06,bachdrko-12}) have found numerous applications in statistics, combinatorics, learning theory, and computational geometry, see e.g.~\cite{Mor-12}
for further details. Lopsided sets represent an ``asymmetric offshoot'' of oriented matroid theory. According to the topological
representation theorem,
oriented matroids can be viewed as regular CW cell complexes decomposing the $(d-1)$-sphere. Lopsided sets on the other hand can be regarded as particular contractible cubical complexes.

In this paper we propose a common generalization of oriented matroids and lopsided sets which is so natural that it is
surprising that it was not discovered much earlier. It also
generalizes such well-known and useful structures as convex geometries and
CAT(0) cube (and zonotopal) complexes. In this generalization, global symmetry and the existence of the zero sign vector, required for oriented
matroids, are replaced by local relative conditions. Analogous to conditional lattices (see~\cite[p. 93]{fuc-63}) and conditional antimatroids (which are particular lopsided sets
\cite{bachdrko-06}), this  motivates the name ``conditional oriented matroids'' (abbreviated: COMs) for these new structures. Furthermore, COMs can be viewed as complexes
whose cells are oriented matroids and which are glued together in a lopsided fashion. To illustrate  the concept of a COM and compare it with similar notions of
oriented matroids and lopsided sets, we continue by describing the geometric model of realizable COMs.

\subsection{Realizable COMs: motivating example}
Let us begin by considering the following familiar scenario of hyperplane arrangements and realizable oriented matroids; compare
with~\cite[Sections 2.1, 4.5]{bjvestwhzi-93} or~\cite[p. 212]{Zie-07}. Given a {\it central arrangement of hyperplanes}
of ${\mathbb R}^d$ (i.e., a finite set $E$ of  $(d-1)$--dimensional linear subspaces of  ${\mathbb R}^d$), the space ${\mathbb R}^d$
is partitioned into open regions and recursively into regions of the intersections of some of the given hyperplanes. Specifically,
we may encode the location of any point from all these regions relative to this arrangement when for each hyperplane one of the
corresponding halfspaces is regarded as positive and the other one as negative. Zero designates location on that hyperplane.
Then the set  $\mathcal L$ of all sign vectors representing the different regions relative to $E$ is the set of covectors of
the oriented matroid of the arrangement $E$. The oriented matroids obtained in this way are called {\it realizable}. If
instead of a central arrangement one considers finite arrangements $E$ of affine hyperplanes (an affine hyperplane is the
translation of a (linear) hyperplane by a vector), then the sets of sign vectors of regions defined by $E$  are known as
{\it realizable affine oriented matroids}~\cite{Karlander-92} and~\cite[p.186]{bachdrko-06}. Since an affine arrangement
on ${\mathbb R}^d$ can be viewed as the intersection of a central arrangement of ${\mathbb R}^{d+1}$ with a translate of
a coordinate hyperplane, each realizable affine oriented matroid can be embedded into a larger realizable oriented matroid.

Now suppose that  $E$ is a central or affine arrangement of hyperplanes of ${\mathbb R}^d$ and $C$ is an open convex set,
which may be assumed to intersect all hyperplanes of $E$ in order to avoid redundancy. Restrict the arrangement pattern to $C$,
that is, remove all sign vectors which represent the open regions disjoint from $C$. Denote the resulting set of sign vectors
by $\covectors(E,C)$ and call it a {\it realizable COM}. Figure~\ref{Figure1a} displays an arrangement comprising two pairs of
parallel lines and a fifth line intersecting the former four lines within the open $4$-gon. Three lines (nos. 2, 3, and 5)
intersect in a common point. The line arrangement defines 11 open regions within the open $4$-gon, which are represented by
their topes, viz. $\pm 1$  covectors. The dotted lines connect adjacent topes and thus determine the tope graph of the arrangement.
This graph is shown in Figure~\ref{Figure1b} unlabeled, but augmented by the covectors of the 14 one-dimensional and 4 two-dimensional faces.

\begin{figure}[htb]
   \centering
   \subfigure[\label{Figure1a}]{
    \includegraphics[width=.4\textwidth]{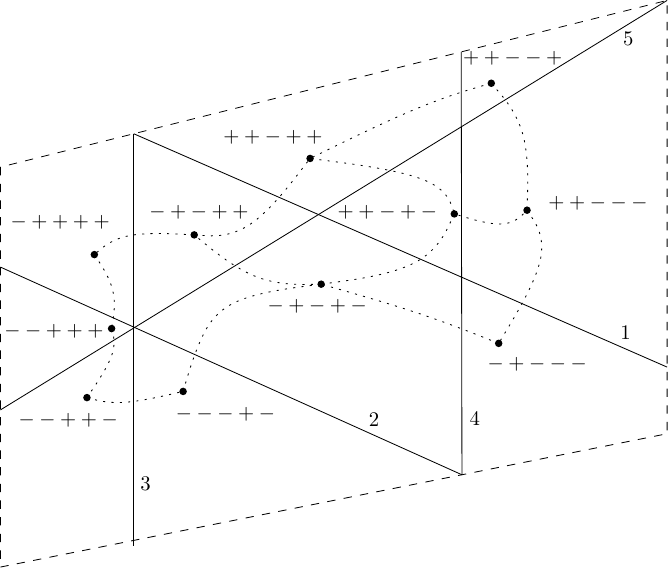}
   }
   \subfigure[\label{Figure1b}]{
    \includegraphics[width=.5\textwidth]{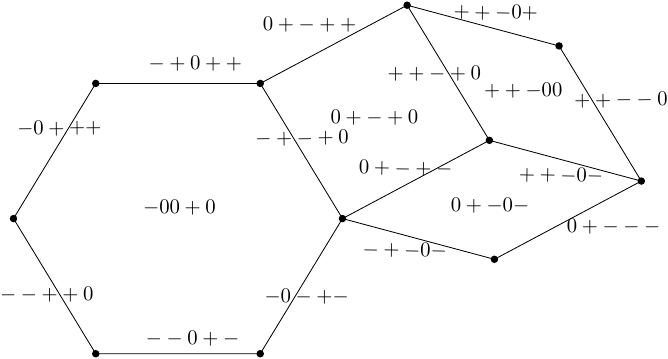}
   }
   \caption{\subref{Figure1a} An arrangement of five lines and its tope graph. \subref{Figure1b} Faces and edges of the tope graph are
   labeled with corresponding covectors. Sign vectors are abbreviated as strings of $+$, $-$, and $0$ and to be read from left to right.}
   \end{figure}

Our model of realizable COMs generalizes realizability of oriented and affine oriented matroids on the one hand and realizability
of lopsided sets on the other hand. In the case of a central arrangement $E$ with $C$ being any open convex set containing the origin
(e.g., the open unit ball or the entire space ${\mathbb R}^d$), the resulting set $\covectors(E,C)$ of
sign vectors coincides with the realizable oriented matroid of $E$. If the arrangement $E$ is affine and $C$ is the entire space, then
$\covectors(E,C)$ coincides with the realizable affine oriented matroid of $E$.
The realizable lopsided sets arise by taking  the (central) arrangement $E$ of all coordinate hyperplanes $E$ restricted to
arbitrary open convex sets $C$ of ${\mathbb R}^d$. In fact, the original definition of
realizable lopsided sets by Lawrence~\cite{la-83} is similar but used instead  an arbitrary (not necessarily open) convex set
$K$ and as regions the closed orthants. Clearly, $K$ can be assumed to be
a polytope, namely the convex hull of points representing the closed orthants meeting $K$. Whenever the polytope $K$ does not
meet a closed orthant then some open neighborhood of $K$ does not meet that orthant either. Since there are only finitely many
orthants, the intersection of these open neighborhoods results in an open set $C$ which has the same intersection pattern with
the closed orthants as $K$. Now, if an open set meets a closed orthant it will also meet the corresponding open orthant,
showing that both concepts of realizable lopsided sets coincide.

\subsection{Properties of realizable COMs} For the general scenario of realizable COMs, we can attempt to identify its basic properties
that are known to hold in oriented matroids.  Let $X$ and $Y$ be sign vectors belonging to $\covectors$, thus designating  regions
represented by two points $x$ and $y$ within $C$ relative to the arrangement $E$; see Figure 2 (compare with Fig. 4.1.1 of~\cite{bjvestwhzi-93}).
Connect the two points by a line segment and choose $\epsilon> 0$ small enough so that the open ball of radius $\epsilon$ around $x$
intersects only those hyperplanes from $E$ on which $x$ lies. Pick any point $w$ from the intersection of this $\epsilon$-ball with
the open line segment between $x$ and $y$. Then the corresponding sign vector $W$ is the {\it composition} $X\circ Y$ as defined by
$$(X\circ Y)_e = X_e  \mbox{ if } X_e\ne 0  \mbox{ and } (X\circ Y)_e=Y_e \mbox{ if } X_e=0.$$
Hence the following rule  is fulfilled:

\medskip\noindent
{\sf (Composition)}~~~~~~  $X\circ Y$ belongs to $\covectors$ for all sign vectors $X$ and $Y$ from $\covectors$. 

\medskip\noindent
If we select instead a point $u$ on the ray from $y$ via $x$ within the $\epsilon$-ball but beyond $x$, then the
corresponding sign vector $U$ has the opposite signs relative to $W$ at the coordinates corresponding to the hyperplanes
from $E$ on which $x$ is located and which do not include the ray from $y$ via $x$. Therefore the following property holds:

\medskip\noindent
{\sf (Face symmetry)}~~~~~~ $X\circ-Y$ belongs to $\covectors$ for all $X,Y$ in $\covectors$.

\medskip\noindent
Next assume that the hyperplane $e$ from $E$ separates $x$ and $y$, that is, the line segment between $x$ and $y$ crosses $e$ at
some point $z$. The corresponding sign vector $Z$ is then zero at $e$ and equals the composition $X\circ Y$ at all coordinates
where $X$ and $Y$ are {\it sign-consistent}, that is, do not have opposite signs:

\medskip\noindent
{\sf (Strong elimination)}~~~~~~ for each pair $X,Y$ in $\covectors$ and for each $e\in E$ with $X_eY_e=-1$ there 
exists $Z \in  \covectors$ such that $Z_e=0$  and  $Z_f=(X\circ Y)_f$  for all $f\in  E$ with $X_fY_f\ne -1$.

\medskip\noindent
Now, the single property of oriented matroids that we have missed in the general scenario is the existence of the zero sign vector, which would correspond to a
non-empty intersection of all hyperplanes from $E$ within the open convex set $C$:

\medskip\noindent
{\sf (Zero vector)}~~~~~~ the zero sign vector ${\bf 0}$ belongs to $\covectors$.

\medskip\noindent
On the other hand, if the hyperplanes from $E$ happen to be the coordinate hyperplanes, then wherever a sign vector $X$ has zero
coordinates, the composition of $X$ with any sign vector from $\{\pm1,0\}^E$ is a sign vector belonging to $\covectors$. This
rule, which is stronger than composition and face symmetry, holds
in lopsided systems, for which the ``tope'' sets are exactly the lopsided sets sensu Lawrence~\cite{la-83}:

\medskip\noindent
{\sf (Ideal composition)}~~~~~~ $X\circ Y\in \covectors$ for all $X\in \covectors$ and all sign vectors $Y$, that is, 
substituting any zero coordinate of a sign vector from $\covectors$ by any other sign yields a sign vector of $\covectors$.

\medskip
\begin{figure}[htb]
\includegraphics[width = .5\textwidth]{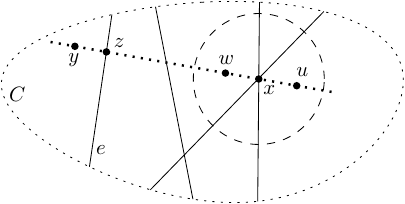}
\caption{Motivating model for the three axioms.}
\label{fig:Fig2}
\end{figure}

In the model of hyperplane arrangements we can retrieve the cells which constitute oriented matroids. Indeed, consider
all non-empty intersections of hyperplanes from $E$ that are minimal with respect to inclusion. Select any sufficiently
small open ball around some point from each intersection. Then the subarrangement of hyperplanes through each of these
points determines regions within these open balls which yield an oriented matroid of sign vectors. Taken together all
these constituents form a complex of oriented matroids, where their intersections are either empty or are faces of the
oriented matroids involved. These complexes are still quite special as they conform to global strong elimination. The
latter feature is not guaranteed in general complexes of oriented matroids, which were called ``bouquets of oriented
matroids''~\cite{DeFu}.

It is somewhat surprising that the generalization of oriented matroids defined by the three fundamental properties of
composition, face symmetry, and strong elimination have apparently not yet been studied systematically. On the other
hand, the preceding discussion shows that the properties of composition and strong elimination hold whenever $C$ is
an arbitrary convex set. We used the hypothesis that the set $C$ be open only for deriving face symmetry. The following
example shows that indeed face symmetry may be lost when $C$ is closed: take two distinct lines in the Euclidean plane,
intersecting in some point $x$ and choose as $C$ a closed halfspace which includes $x$ and the entire $++$ region but is disjoint from the $--$
region. Then $+-, +0, ++, 0+, -+$, and $00$ comprise the sign vectors of the regions within $C$, thus violating face
symmetry. Indeed, the obtained system can be regarded as a lopsided system with an artificial zero added. On the other
hand, one can see that objects obtained this way are realizable \emph{oriented matroid polyhedra}~\cite[p. 420]{bjvestwhzi-93}.

\subsection{Structure of the paper}
In Section~\ref{sec:basic_axioms} we will continue by formally introducing the systems of sign vectors considered in this paper.
In Section~\ref{sec:minors} we prove that COMs are closed under minors and simplification, thus sharing this fundamental
property with oriented matroids. We also introduce the fundamental concepts of fibers and faces of COMs, and show that
faces of COMs are OMs. Section~\ref{sec:topes} is dedicated to topes and tope graphs of COMs and we show that both these
objects uniquely determine a COM.  Section~\ref{sec:mingen} is devoted  to  characterizations of minimal systems of
sign-vectors which  generate a given COM by composition. In Section~\ref{sec:cocircuit} we extend these characterizations
and, analogously to oriented matroids, obtain a characterization of COMs in terms of cocircuits. In Section~\ref{sec:hyperplanes}
we define carriers, hyperplanes, and halfspaces, all being COMs naturally contained in a given COM. We present a characterization
of COMs in terms of these substructures. In Section~\ref{sec:amalgam} we study decomposition and amalgamation procedures for
COMs and show that every COM can be obtained by successive amalgamation of oriented matroids. In Section~\ref{sec:Euler}, we
extend the Euler-Poincar\'e formula from OMs to COMs and characterize lopsided sets in terms of a particular variant of it.
In Section~\ref{sec:ranking} as a resuming example we study the COMs provided by the ranking extensions -- aka weak
extensions -- of a partially ordered set and illustrate the operations and the results of the paper on them.   In
Section~\ref{sec:complexes} we consider a topological approach to COMs and study them as complexes of oriented
matroids. In particular, we show that non-positively curved Coxeter zonotopal complexes give rise to COMs.  We
close the paper with several  concluding remarks and two conjectures in Section~\ref{sec:conclude}.

\section{Basic axioms}\label{sec:basic_axioms}
We follow the standard oriented matroids notation from~\cite{bjvestwhzi-93}. Let $E$ be a non-empty finite (ground)
set and let $\covectors$ be a non-empty {\it set of sign vectors}, i.e.,
maps from $E$ to $\{\pm 1,0\} = \{-1,0,+1\}$. The elements of $\covectors$ are also referred to as \emph{covectors} and
denoted by capital letters $X, Y, Z$, etc.  For $X \in \covectors$, the subset $\underline{X} = \{e\in E: X_e\neq 0\}$
is called the \emph{support} of $X$ and  its complement  $X^0=E\setminus \underline{X}=\{e\in E: X_e=0\}$ the \emph{zero set} of $X$
(alias the \emph{kernel} of $X$). Simlaly, we denote $X^+=\{e\in E: X_e=+\}$ and $X^-=\{e\in E: X_e=-\}$. We can regard a sign vector $X$ as the incidence vector of a $\pm1$ signed subset $\underline{X}$ of
$E$ such that to each element of $E$ one element of the signs $\{\pm1, 0\}$ is assigned. We denote by $\leq$  the product ordering
on $\{ \pm 1,0\}^E $ relative to the standard
ordering of signs with $0 \leq -1$ (sic!) and $0 \leq +1$.

For $X,Y\in \covectors$, we call $S(X,Y)=\{f\in E: X_fY_f=-1\}$
the \emph{separator} of $X$ and $Y$. The elements of $S(X,Y)$ are said to \emph{separate} $X$ and $Y$.
If the separator is empty, then $X$ and $Y$ are said to be {\it sign-consistent}.
In particular, this is the case when $X$ is below $Y$, that is, $X\le Y$ holds.  The \emph{composition}
of $X$ and $Y$ is the sign vector $X\circ Y,$ where
$(X\circ Y)_e = X_e$  if $X_e\ne 0$  and  $(X\circ Y)_e=Y_e$  if $X_e=0$. Note that $X\leq X\circ Y$
for all sign vectors $X,Y$.

Given a set of sign vectors $\covectors$, its \emph{topes} are the maximal elements of $\covectors$
with respect to $\leq$. Further let
$$\uparr \covectors:=\{ Y \in \{ \pm 1,0\}^E: X \leq Y \text{ for some } X \in  \covectors\}=\{X\circ W: X\in\covectors \text{ and } W\in \{ \pm 1,0\}^E\}$$
be the \emph{upset} of $\covectors$ in the ordered set $(\{ \pm 1,0\}^E,\leq)$.

If a set of sign  vectors is closed with respect to $\circ$, then the resulting idempotent semigroup (indeed a left regular band
alias skew semilattice~\cite{MargSF,MargarXiv}) is called the \emph{braid semigroup}, see e.g.~\cite{bjorner08}. The composition operation
naturally occurs also elsewhere:
for a single coordinate, the composite $x\circ y$ on $\{\pm 1,0\}$ is actually derived as the term $t(x,0,y)$ (using $0$ as a constant)
from the ternary discriminator $t$ on $\{\pm 1,0\}$, which is defined by $t(a,b,c)=a$ if $a\ne b$ and $t(a,b,c)=c$ otherwise.
Then in this context of algebra and logic, ``composition'' on the product $\{\pm 1,0\}^E$ would rather be referred to
as a ``skew Boolean join''~\cite{bignall-91}.

We continue with the formal definition of the main axioms as motivated and discussed in the previous section.

\medskip\noindent
{\bf Composition:}
\begin{itemize}
\item[{\bf (C)}] $X\circ Y \in  \covectors$  for all $X,Y \in  \covectors$.
\end{itemize}

Condition (C) is taken from the list of axioms for oriented matroids. Since $\circ$ is associative, arbitrary
finite compositions can be written without bracketing $X_1\circ\ldots \circ X_k$ so
that (C) entails that they all belong to $\covectors$. Note that contrary to a convention sometimes made in
oriented matroids we do not consider compositions over an empty index set, since this would
imply that the zero sign vector belonged to $\covectors$. We highlight condition (C) here although it
will turn out to be a consequence of another axiom specific in this context. The reason is that we
will later use several weaker forms of the axioms which are no longer consequences from one another.

\medskip\noindent
{\bf Strong elimination:}
\begin{itemize}
\item[{\bf (SE)}]  for each pair $X,Y\in\covectors$ and for each $e\in  S(X,Y)$ there exists $Z \in  \covectors$ such that
$Z_e=0$  and  $Z_f=(X\circ Y)_f$  for all $f\in E\setminus S(X,Y)$.
\end{itemize}

Note that $(X\circ Y)_f=(Y\circ X)_f$ holds exactly when $f\in E\setminus S(X,Y).$ Therefore the sign vector $Z$ provided by (SE) serves both ordered pairs $X,Y$ and $Y,X$.

Condition (SE) is one of the axioms for covectors of oriented matroids
and is implied by the property of route systems in lopsided sets, see Theorem 5 of~\cite{la-83}.

\bigskip\noindent
{\bf Symmetry:}
\begin{itemize}
\item[{\bf (Sym)}] $-\covectors=\{ -X: X\in \covectors\}=\covectors,$ that is, $\covectors$ is closed under sign reversal.
\end{itemize}

Symmetry restricted to zero sets of covectors (where corresponding supports are retained) is dubbed:

\medskip\noindent
{\bf Face symmetry:}
\begin{itemize}
\item[{\bf (FS)}] $X\circ -Y \in  \covectors$  for all $X,Y \in  \covectors$.
\end{itemize}

This condition can also be expressed by requiring that for each pair $X,Y$ in $\covectors$ there exists $Z\in \covectors$ with $X\circ Z=Z$ such that $X=\frac{1}{2}(X\circ Y+X\circ Z)$.
Face symmetry trivially implies (C) because by (FS) we first get $X\circ -Y\in\covectors$ and then $X\circ Y= (X\circ -X)\circ Y= X\circ -(X\circ -Y)$ for all $X,Y\in \covectors$.

\medskip\noindent
{\bf Ideal composition:}
\begin{itemize}
\item[{\bf (IC)}] $\uparr\covectors=\covectors$.
\end{itemize}
Notice that (IC) implies (C) and (FS).
We are now ready to define the main objects of our study:
\begin{definition}\label{def}
 A system of sign vectors  $(E,\covectors)$ is called a:
 \begin{itemize}
  \item\emph{strong elimination system} if $\covectors$ satisfies (C) and (SE),
  \item\emph{conditional oriented matroid (COM)} if  $\covectors$ satisfies (FS) and (SE),
  \item\emph{oriented matroid (OM)} if $\covectors$ satisfies (C), (Sym), and (SE),
   \item\emph{lopsided system} if $\covectors$ satisfies (IC) and (SE).
 \end{itemize}
\end{definition}

For oriented matroids one can replace (C) and (Sym) by (FS) and

\medskip\noindent
{\bf Zero vector:}
\begin{itemize}
\item[{\bf (Z)}]  the zero sign vector ${\bf 0}$ belongs to $\covectors$.
\end{itemize}

Notice that the axiom (SE) can be somewhat weakened in the presence of (C), i.e., in particular in Definition~\ref{def}. If
(C) is true in the system $(E,\covectors)$, then for $X,Y \in  \covectors$ we have
$X\circ Y = (X\circ Y)\circ (Y\circ X)$,  $\underline{X\circ Y} = \underline{Y\circ X} = \underline{X}\cup \underline{Y}$, 
and also for the separators we have $S(X\circ Y,Y\circ X) = S(X,Y)$.

Therefore, if (C) holds, we may substitute (SE) by
\begin{itemize}
\item[{\bf (SE$^=$)}] for each pair $X, Y\in \covectors$ with $\underline{X} = \underline{Y}$ and for each $e\in  S(X,Y)$
there exists $Z\in \covectors$ such that $Z_e=0$  and  $Z_f=(X\circ Y)_f$  for all $f\in  E\setminus S(X,Y)$,
\end{itemize}

The axioms (C), (FS), (SE$^=$) (plus a fourth condition) were used by Karlander~\cite{Karlander-92} in his study 
of affine oriented matroids that are embedded as ``halfspaces" (see Section~\ref{sec:hyperplanes} below) of oriented matroids.

\section{Minors, fibers, and faces}\label{sec:minors}
In the present technical section we show that the class of COMs is closed under taking minors, defined as for oriented matroids. 
We use this to establish that simplifications and semisimplifications of COMs are minors of COMs and therefore COMs. We also introduce 
fibers and faces of COMs, which will be of importance for the rest of the paper.

Let $(E,\covectors)$ be a COM and $A\subseteq E$. Given a sign vector $X\in\{\pm 1, 0\}^E$  by $X\setminus A$ we refer to the \emph{restriction} 
of $X$ to $E\setminus A$, that is $X\setminus A\in\{\pm1, 0\}^{E\setminus A}$ with $(X\setminus A)_e=X_e$ for all $e\in E\setminus A$. 
The \emph{deletion} of $A$ is defined as $(E\setminus A,\covectors\setminus A)$, where $\covectors\setminus A:=\{X\setminus A:  X\in\covectors\}$. 
The \emph{contraction} of $A$ is defined as $(E\setminus A,\covectors/ A)$, where 
$\covectors/ A:=\{X\setminus A: X\in\covectors\text{ and }\underline{X}\cap A=\varnothing\}$. If a system of sign vectors arises by 
deletions and contractions from another one it is said to be \emph{minor} of it.

\begin{lemma}\label{lem:minorclosed}
The properties {\normalfont (C)}, {\normalfont (FS)}, and {\normalfont (SE)} are all closed under taking minors.
In particular, if $(E,\covectors)$ is a COM and $A\subseteq E$, then $(E\setminus A,\covectors\setminus A)$ and 
$(E\setminus A,\covectors/ A)$ are COMs as well.
\end{lemma}

\begin{proof}
We first prove that $(E\setminus A,\covectors\setminus A)$ is a COM. To see (C) and (FS) let $X\setminus A,Y\setminus A\in \covectors\setminus A$. 
Then $X\circ (\pm Y)\in \covectors$ and $(X\circ (\pm Y))\setminus A= X\setminus A\circ (\pm Y\setminus A)\in\covectors\setminus A$.

To see (SE) let $X\setminus A,Y\setminus A\in \covectors\setminus A$ and $e$ an element separating $X\setminus A$ and $Y\setminus A$.  
Then there is $Z\in\covectors$ with $Z_e=0$ and $Z_f=(X\circ Y)_f$ for all $f\in E\setminus S(X,Y)$. Clearly,
$Z\setminus A\in \covectors\setminus A$ satisfies (SE) with respect to $X\setminus A,Y\setminus A$.

Now, we prove that $(E\setminus A,\covectors/ A)$ is a COM. Let $X\setminus A,Y\setminus A\in \covectors/A$, i.e.,
$\underline{X}\cap A=\underline{Y}\cap A=\varnothing$. Hence $\underline{X\circ (\pm Y)}\cap A=\varnothing$ and
therefore $X\setminus A\circ (\pm Y\setminus A)\in \covectors/A$, proving (C) and (FS).

To see (SE) let $X\setminus A,Y\setminus A\in \covectors/ A$ and $e$ an element separating $X\setminus A$
and $Y\setminus A$. 
Then there is $Z\in\covectors$ with $Z_e=0$ and $Z_f=(X\circ Y)_f$ for all $f\in E\setminus S(X,Y)$.
In particular, if $X_f=Y_f=0$, then $Z_f=0$. Therefore, $Z\setminus A\in \covectors/A$ and it satisfies (SE).
\end{proof}

\begin{lemma}\label{lem:minorcommute}
If $(E,\covectors)$ is a system of sign vectors and $A, B\subseteq E$ with $A\cap B=\varnothing$, then
$(E\setminus (A\cup B),(\covectors\setminus A)/B)=(E\setminus (A\cup B),(\covectors/B)\setminus A)$.
\end{lemma}

\begin{proof}
It suffices to prove this for $A$ and $B$ consisting of single elements $e$ and $f$, respectively.
Now $X\setminus\{e,f\}\in (\covectors\setminus \{e\})/\{f\}$ if and only if $X\in\covectors$ with $X_f=0$
which is equivalent to $X\setminus\{e\}\in\covectors\setminus\{e\}$ with $(X\setminus\{e\})_f=0$.
This is, $X\setminus\{e,f\}\in (\covectors/\{f\})\setminus \{e\}$.
\end{proof}

Next, we will define simple and semisimple  systems of sign vectors. These are motivated by the hyperplane model for COMs,
that possesses additional properties, reflecting that we have a set of hyperplanes rather than a multiset and that the given
convex set is open. This is also motivated by the requirement of defining systems of sign vectors not containing coloops and
parallel elements, which is relevant, for example, for the identifications of topes.

A {\it coloop} of $(E,\covectors)$ is an element $e\in E$ such that $X_e=0$ for all $X\in \covectors$. Hence
$(E,\covectors)$ does not have coloops if and only if for  each element $e$,  there exists a
covector $X$ with $X_e \neq 0$. Two elements $e,e'\in E$ of $(E,\covectors)$ are \emph{parallel}, denoted $e\parallel e'$,
if either $X_e=X_{e'}$ for all $X\in\covectors$ or  $X_e=-X_{e'}$ for all $X\in\covectors$. It is easy to see that $\parallel$
is an equivalence relation. The condition that $(E,\covectors)$ does not contain parallel elements can be expressed
by the requirement that for each pair $e\neq f$ in $E$,  there exist $X,Y \in \covectors$ with $X_e\neq X_f$ and $Y_e\neq -Y_f$.

Simple systems are defined by two axioms which are slightly stronger than those which ensure that  the absence of
coloops and parallel elements. We call the system $(E,\covectors)$ {\it simple} if it satisfies the following non-redundancy axioms:

\medskip\noindent
{\bf Simplicity:}
\begin{itemize}
\item[{\bf (N1)}]  for each $e \in E$,  $\{\pm 1,0 \}=\{X_e: X\in \covectors\}$;
\item[{\bf (N2)}]  for each pair $e\neq f$ in $E$,  there exist $X,Y \in \covectors$ with $\{X_eX_f, Y_eY_f\}=\{\pm1\}$.
\end{itemize}

We will also consider the weaker notion of semisimple systems, which are the simple systems when restricted
to the set $$E_{\pm}:=\{ e\in E: \{ X_e: X\in \covectors\}\ne \{ +1\},\{ -1\}\}$$ of those elements $e$ at which $\covectors$
is not non-zero constant.  We call the system $(E,\covectors)$ {\it semisimple} if it satisfies the following restricted
non-redundancy axioms:

\medskip\noindent
{\bf Semisimplicity:}
\begin{itemize}
\item[{\bf (RN1)}]  for each $e \in E_{\pm}$,  $\{\pm 1,0 \}=\{X_e: X\in \covectors\}$;
\item[{\bf (RN2)}]  for each pair $e\neq f$ in $E_{\pm}$,  there exist $X,Y \in \covectors$ with $\{X_eX_f, Y_eY_f\}=\{\pm1\}$.
\end{itemize}

Let $E_0:=\{e\in E: X_e=0$ for all $X\in\covectors\}$ be the set of all coloops. Condition (RN1) yields that $E_0=\varnothing$.
The condition that there are no coloops is relevant for the identification of topes. Recall that a tope of $\covectors$  is any covector $X$
that is maximal with respect to the standard sign ordering defined above. In the presence of (C), the covector $X$ is a tope precisely when
$X\circ Y = X$ for all $Y\in \covectors$, that is, for each
$e\in E$ either $X_e\in \{\pm 1\}$ or $Y_e = 0$ for all $Y\in \covectors$. In particular, if both (C) and $E_0=\varnothing$
hold, then the topes are exactly the covectors with full support $E$. Notice also that condition (N2) yields that there are no parallel elements. 
Consequently, simple systems do not contain coloops, parallel elements,  and their topes are the covectors with full support (while semisimple 
systems satisfy the first and the third conditions).

Further put
$$E_1:=\{ e \in E: \#\{ X_e: X\in \covectors\}=1\}= E_0\cup (E\setminus E_{\pm}),$$
$$E_2:=\{e \in E: \#\{ X_e: X\in \covectors\}=2\}.$$
The sets $E_0\cup E_2$ and $E_1\cup E_2$ comprise the positions at which $\covectors$ violates (RN1) or (N1), respectively.
Hence the deletions $(E\setminus(E_0\cup E_2), \covectors\setminus (E_0\cup E_2))$ and $(E\setminus(E_1\cup E_2), \covectors\setminus (E_1\cup E_2))$
constitute the canonical transforms of $(E,\covectors)$ satisfying (RN1) and (N1), respectively. One equivalence class of the parallel
relation $\|$ restricted to $\covectors\setminus (E_0\cup E_2)$ is $E\setminus E_{\pm}$, i.e., it comprises the (non-zero) constant positions.
Exactly this class gets removed when one restricts $\|$ further to $\covectors\setminus (E_1\cup E_2).$  Selecting a transversal for the
classes of $\|$ on $\covectors\setminus (E_1\cup E_2)$, which arbitrarily picks exactly one element from each class and deletes all others,
results in a simple system. Restoring the entire class $E\setminus E_{\pm}$ then yields a semisimple system. We refer to these canonical
transforms as to the {\it simplification} and {\it semisimplification} of $(E,\covectors)$, respectively. Then from Lemma~\ref{lem:minorclosed} we obtain:

\begin{lemma}
The semisimplification of a system $(E,\covectors)$ of sign vectors is a semisimple minor and the simplification
is a simple minor, unique up to sign reversal on subsets of $E_{\pm}$. Either system is a COM whenever $(E,\covectors)$ is.
\end{lemma}

For a system $(E,\covectors)$, a \emph{fiber} relative to some $X\in \covectors$ and $A\subseteq E$ is a set of sign vectors defined by
$${\mathcal R}=\{ Y\in \covectors: Y\setminus A=X\setminus A\}.$$
We say that such a fiber is {\it topal} if $(X\setminus A)^0=\varnothing$, that is, $X\setminus A$ is a tope of the minor
$(E\setminus A, \covectors\setminus A)$. $\mathcal R$ is a {\it face} if $X$ can be chosen so that $X^0=A$, whence faces are
topal fibers. Note that the entire system $(E,\covectors)$ can be regarded both as a topal fiber
and as the result of an empty deletion or contraction. If $(E,\covectors)$ satisfies (C), then the fiber
relative to $A:=X^0$, alias $X$-{\it face}, associated with a sign vector $X$ can be expressed in the form
$$F(X):=\{ X\circ Y: Y\in \covectors\}=\covectors \hspace{2pt}\cap \uparr\{ X\}.$$
If $S(V,W)$ is non-empty for $V,W\in \covectors$, then the corresponding faces $F(V)$ and $F(W)$ are disjoint.
Else, if $V$ and $W$ are sign-consistent, then  $F(V)\cap F(W)=F(V\circ W)$. In particular $F(V)\subseteq F(W)$
is equivalent to $V\in F(W)$, that is, $W\le V$. The ordering of faces by inclusion thus reverses the sign ordering.
The following observations are straightforward and recorded here for later use:

\begin{lemma}\label{lem:fiber} If $(E,\covectors)$ is a strong elimination system or a COM, respectively, then so are
all fibers of $(E,\covectors)$. If $(E,\covectors)$ is semisimple, then so is every topal fiber. If $(E,\covectors)$
is a COM, then for any $X\in\covectors$ the minor $(E\setminus\underline{X},F(X)\setminus \underline{X})$ corresponding
to the face $F(X)$ is an OM, which is simple whenever $(E,\covectors)$ is semisimple.
\end{lemma}

\section{Tope graphs}\label{sec:topes}

One may wonder whether and how the topes of a semisimple COM $(E,\covectors)$ determine and generate $\covectors$.
We cannot avoid using face symmetry because one can turn every COM which is not an OM into a strong elimination system by
adding the zero vector to the system, without affecting the topes. The following result for simple oriented matroids was
first observed  by Mandel (unpublished), see~\cite[Theorem 4.2.13]{bjvestwhzi-93}.

\begin{proposition} \label{topes} Every semisimple COM $(E,\covectors)$ is uniquely determined by its set of topes.
\end{proposition}

\begin{proof} We proceed by induction on $\#E$. For a single position the assertion is trivial. So assume $\#E\ge 2$.
Let $\covectors$ and $\covectors'$ be two COMs on $E$ sharing the same set of topes.
Then deletion of any $g\in E$ results in two COMs with equal tope sets, whence $\covectors'\setminus g=\covectors\setminus g$
by the induction hypothesis. Suppose that there exists some $W\in \covectors'\setminus \covectors$ chosen with $W^0$ as small
as possible. Then $\#W^0>0$. Take any $e\in W^0$. Then as $\covectors'\setminus e=\covectors\setminus e$ by semisimplicity
there exists a sign vector $V$ in $\covectors$ such that $V\setminus e=W\setminus e$ and $V_e\ne 0.$ Since $V^0\subset W^0$,
we infer that $V\in \covectors'$ by the minimality choice of $W$. Then, by (FS) applied to $W$ and $V$ in $\covectors'$,
we get $W\circ -V\in \covectors'$. This sign vector also belongs to $\covectors$ because $\#(W\circ -V)^0=\# V^0<\# W^0$.
Finally, apply (SE) to the pair $V, W\circ -V$ in $\covectors$ relative to $e$ and obtain $Z=W\in \covectors$, in conflict
with the initial assumption.
\end{proof}

The {\it tope graph} of a semisimple COM on $E$ is the graph with all topes as its vertices where two topes are adjacent
exactly when they differ in exactly one coordinate. In other words, the tope graph is the subgraph of the $\#E$-dimensional
hypercube with vertex set $\{\pm1\}^E$ induced by the tope set. Isometry means that the internal distance in the subgraph
is the same as in the hypercube. Isometric subgraphs of the hypercube are often referred to as a {\it partial cubes}~\cite{ham-11}.
For tope graphs of oriented matroids the next result was first proved in~\cite{lv-80}

\begin{proposition} \label{OMCtoTopeGraph}  The tope graph of a semisimple strong elimination system $(E,\covectors)$ is a
partial cube in which the edges correspond to the sign vectors of $\covectors$  with singleton zero sets.
\end{proposition}

\begin{proof} If $X$ and $Y$ are two adjacent topes, say, differing at position $e\in E$, then the vector $Z\in \covectors$
provided by (SE) for this pair relative to $e$ has $0$ at $e$ and coincides with $X$ and $Y$ at all other positions.
By way of contradiction assume that now $X$ and $Y$ are two topes which cannot be connected by a path in the tope
graph of length $\#S(X,Y)=k>1$ such that $k$ is as small as possible. Then the interval $[X,Y]$ consisting of all
topes on shortest paths between $X$ and $Y$ in the tope graph comprises only $X$ and $Y$. For $e \in  S(X,Y)$ we
find some $Z \in  \covectors$ such that $Z_e = 0$  and  $Z_g = X_g$  for all $g \in  E \setminus S(X,Y)$ by (SE).
If there exists $f \in  S(X,Y) \setminus \{e\}$ with $Z_f \neq 0$, then $Z\circ X$ or $Z\circ Y$ is a tope different
from $X$ and $Y$, but contained in $[X,Y]$, a contradiction.

If $Z_f = 0$ for all $f \in  S(X,Y) \setminus \{e\}$, then by (RN2) there is $W\in\mathcal{L}$ with
$0\neq W_eW_f\neq X_eX_f\neq 0$. We conclude that  $Z\circ W\circ Y$ is a tope different from $X$ and $Y$ but
contained in $[X,Y]$, a contradiction. This concludes the proof.
\end{proof}

Isometric embeddings of partial cubes into hypercubes are unique up to automorphisms of the hosting
hypercube~\cite[Proposition 19.1.2]{DeLa} (and addition of superfluous dimensions). Hence,
Propositions~\ref{topes} and~\ref{OMCtoTopeGraph} together imply
the following result, which generalizes a similar result of~\cite{bjedzi-90} for tope graphs of OMs:

\begin{proposition}\label{TopeGraphtoOMC} A simple COM is determined by its tope graph up to reorientation.
\end{proposition}

\section{Minimal generators of strong elimination systems}\label{sec:mingen}
We have seen in the preceding section that a COM is determined by its tope set. There is a
more straightforward way to generate any strong elimination system from bottom to top by
taking suprema. This generation process involves only some weaker forms of the axioms
(C) and (SE).

Let $(E,\covectors)$ be a system of sign vectors. Given $X,Y\in \covectors$ consider the
following set of sign vectors which partially ``conform'' to $X$ relative to subsets $A\subseteq S(X,Y)$:
\begin{align*}
{\mathcal W}_A(X,Y)&=\{ Z\in \covectors: Z^+\subseteq X^+\cup Y^+, Z^-\subseteq X^-\cup Y^-, \mbox{ and } S(X,Z)\subseteq E\setminus A\}\\
&=\{ Z\in \covectors: Z_g\in \{ 0,X_g,Y_g\} \mbox{ for all } g\in E, \mbox{ and } Z_h\in \{ 0,X_h\} \mbox{ for all } h\in A\}.
\end{align*}
For $A=\varnothing$ we use the short-hand ${\mathcal W}(X,Y)$, i.e.,
\begin{align*}
{\mathcal W}(X,Y)&=\{ Z\in \covectors: Z^+\subseteq X^+\cup Y^+, Z^-\subseteq X^-\cup Y^-\}
\end{align*}
and for the maximum choice $A\supseteq S(X,Y)$ we write ${\mathcal W}_{\infty}(X,Y)$, i.e.,
\begin{align*}
{\mathcal W}_{\infty}(X,Y)&=\{ Z\in {\mathcal W}(X,Y): S(X,Z)=\varnothing\}.
\end{align*}

Trivially, $X,X\circ Y\in {\mathcal W}_A(X,Y)$ and ${\mathcal W}_B(X,Y)\subseteq {\mathcal W}_A(X,Y)$ for $A\subseteq B\subseteq E$. Note that
$S(X,Z)\subseteq S(X,Y)$ for all $Z\in {\mathcal W}(X,Y)$. Each set ${\mathcal W}_A(X,Y)$ is closed under composition (and trivially is a downset with
respect to the sign ordering). For, if $V,W\in {\mathcal W}_A(X,Y),$ then $(V\circ W)^+\subseteq V^+\cup W^+$ and $(V\circ W)^-\subseteq V^-\cup W^-$
holds trivially, and further, if $e\in S(X,V\circ W)$, say, $e\in X^+$ and $e\in (V\circ W)^-\subseteq V^-\cup W^-$, then
$e\in S(X,V)$ or $e\in S(X,W),$ that is,
$$S(X,V\circ W)\subseteq S(X,V)\cup S(X,W)\subseteq E\setminus A.$$

Since each of the sets ${\mathcal W}_A(X,Y)$ is closed under composition, we may take the composition of all sign vectors
in ${\mathcal W}_A(X,Y)$. The result may depend on the order of the constituents.

Some features of strong elimination are captured by weak elimination:

\medskip\noindent
\begin{itemize}
\item[{\bf (WE)}]  for each pair $X,Y\in\covectors$ and $e\in  S(X,Y)$ there exists $Z \in  {\mathcal W}(X,Y)$ with $Z_e=0$.
\end{itemize}

\medskip\noindent Condition (WE) is in general weaker than (SE): consider, e.g., the four sign vectors $++,+-,--,00;$ the zero vector $Z$ would
serve all pairs $X,Y$ for (WE) but for $X=++$ and $Y=+-$ (SE) would require the existence of $+0$ rather than $00$. In the presence
of (IC), the strong and the weak versions of elimination are equivalent, that is, lopsided systems are characterized by
(IC) and (WE)~\cite{bachdrko-12}. With systems satisfying (WE) one can generate lopsided systems by taking the upper sets:

\begin{proposition}[\cite{bachdrko-12}]\label{lopsided_enveloppe} If $(E,{\mathcal K})$ is a system of sign vectors which
satisfies {\normalfont (WE)}, then $(E, \uparr{\mathcal K})$ is a lopsided system.
\end{proposition}

\begin{proof} We have to show that (WE) holds for $(E,\uparr{\mathcal K}).$ For $X,Y\in \uparr{\mathcal K}$ and some
element $e$ in $S(X,Y)$, pick $V,W\in {\mathcal K}$ with $V\le X$ and $W\le Y$. If $e\in S(V,W),$ then by (WE)
in $\mathcal K$ one obtains some $U\in\uparr{\mathcal K}$ such that $U_e=0$ and $U_f\le V_f\circ W_f\le X_f\circ Y_f$ for
all $f\in E\setminus S(X,Y)$. Then the sign vector $Z$ defined by $Z_g:=U_g$ for all $g\in S(X,Y)$ and $Z_f:=X_f\circ Y_f$
for all $f\in E\setminus S(X,Y)$ satisfies $U\le Z$ and hence belongs to $\uparr{\mathcal K}$. If $e\notin S(V,W)$, then
$V_e=0$, say. Define a sign vector $Z$ similarly as above: $Z_g:=V_g$ for $g\in S(X,Y)$ and $Z_f:=X_f\circ Y_f\ge V_f$ for
$f\in E\setminus S(X,Y)$. Then $Z\in \uparr{\mathcal K}$ is as required.
\end{proof}

This proposition applied to a COM $(E,\covectors)$ yields an associated lopsided systems $(E,\uparr\covectors)$ having
the same minimal sign vectors as $(E,\covectors)$. This system is referred to as the {\it lopsided envelope} of $(E,\covectors).$
In contrast to (SE) and (SE$^{=}$), the following variant of strong elimination allows us to treat the positions
$f\in E\setminus S(X,Y)$ one at a time:
\medskip\noindent
\begin{itemize}
\item[{\bf (SE1)}] for each pair $X,Y\in\covectors$ and $e\in  S(X,Y)$ and $f\in  E\setminus S(X,Y)$ there exists
$Z \in  {\mathcal W}(X,Y)$ such that $Z_e=0,$ and $Z_f=(X\circ Y)_f$.
\end{itemize}

Nevertheless, under composition axiom {\normalfont (C)}, all these variants of {\normalfont (SE)}
are equivalent:
\begin{lemma}\label{lem:auxilliary} Let $(E,\covectors)$ be a system of sign vectors which satisfies {\normalfont (C)}. Then
all three variants {\normalfont (SE)}, {\normalfont (SE$^=$)}, and {\normalfont (SE1)}
are equivalent.
\end{lemma}

\begin{proof} The equivalence of (SE) and (SE$^=$) in the presence of (C) was argued at the end of Section~\ref{sec:basic_axioms}.
Trivially, (SE) implies (SE1). Conversely, if (SE1) holds, then for every $e\in S(X,Y)$ we obtain a set
$\{ Z^{(\{e\},f)}: f\in E\setminus S(X,Y)\}$
of solutions, one for each $f$. Then the composition in any order of these solutions  yields a solution $Z$ for
(SE), because $Z^{(\{e\},f)}\le Z$ for all $f\in E\setminus S(X,Y)$ and $Z_e=0,$ whence $Z_f=(X\circ Y)_f$ for
all $f\in E\setminus S(X,Y)$ and $Z_e=0$.
\end{proof}

Since strong elimination captures some features of composition, one may wonder whether (C) can be
somewhat weakened in the presence of (SE) or (SE1). Here suprema alias conformal compositions
come into play:

\medskip\noindent
\begin{itemize}
\item[{\bf (CC)}] $X\circ Y \in  \covectors$  for all $X,Y \in  \covectors$ with $S(X,Y)=\varnothing$.
\end{itemize}
Recall that $X$ and $Y$ are sign-consistent, that is, $S(X,Y)=\varnothing$ exactly when $X$ and $Y$ commute: $X\circ Y=Y\circ X$.
We say that a composition $X^{(1)}\circ\ldots\circ X^{(n)}$ of sign vectors is {\it  conformal} if it constitutes the
supremum of $X^{(1)},\ldots,X^{(n)}$ with respect to the sign ordering. Thus, $X^{(1)},\ldots,X^{(n)}$ commute exactly
when they are bounded from above by some sign vector, which is the case when the set of all $X_e^{(i)}$ $(1\le i\le n)$
includes at most one non-zero sign (where $e$ is any fixed element of $E$). If we wish to highlight this property we denote
the supremum of $X^{(1)},\ldots,X^{(n)}$ by $\bigodot_{i=1}^nX^{(i)}$ or $X^{(1)}\odot\ldots \odot X^{(n)}$ (instead
of $X^{(1)}\circ \ldots \circ X^{(n)}$). Clearly the conformal property is Helly-type in the sense that a set of
sign vectors has a supremum if each pair in that set does.

Given any system $\mathcal{K}$ of sign vectors on $E$ define $\bigodot\mathcal{K}$ as the set of all (non-empty)
suprema of members from $\mathcal{K}$. We say that a system $(E,\mathcal{K})$ of sign vectors {\it generates} a
system $(E,\covectors)$ if $\bigodot\mathcal{K}=\covectors$. We call a sign
vector $X\in\covectors$   ({\it supremum-}){\it irreducible} if it does not equal the (non-empty!)
conformal composition of any sign vectors from $\covectors$ different from $X$. Clearly, the irreducible sign
vectors of $\covectors$ are unavoidable when generating $\covectors$. We denote the set of irreducibles of
$\covectors$ by $\mathcal{J}=\mathcal{J}(\covectors)$.

\begin{theorem} \label{thm:cc_ses} Let $(E,\covectors)$ be a system of sign vectors. Then the following conditions are equivalent:
\begin{itemize}
 \item[(i)] $(E,\covectors)$ is a strong elimination system;
 \item[(ii)] $\covectors$ satisfies {\normalfont (CC)} and {\normalfont (SE1)};
 \item[(iii)] $\covectors$ satisfies {\normalfont (CC)} and some set $\mathcal{K}$ with
 $\mathcal{J}\subseteq \mathcal{K}\subseteq\covectors$ satisfies {\normalfont (SE1)}.
 \item[(iv)] $\covectors$ satisfies {\normalfont (CC)} and its set $\mathcal{J}$ of
 irreducibles satisfies {\normalfont (SE1)}.
\end{itemize}
\end{theorem}

\begin{proof} The implication (i)$\implies$(ii) is trivial.
Now, to see (ii)$\implies$(iv) let $(E,\covectors)$ satisfy (CC) and (SE1). For $X,Y\in {\mathcal J},$ $e\in S(X,Y),$
and $f\in E\setminus S(X,Y)$ we first obtain $Z\in \covectors$ with $Z_e=0$ and $Z_f=(X\circ Y)_f$. Since $Z$ is
the supremum of some $Z^{(1)},\ldots,Z^{(n)}$ from $\mathcal J$, there must be an index $i$ for which $Z^{(i)}_f=Z_f$
and trivially $Z_e^{(i)}=0$ holds. Therefore ${\mathcal J}$ satisfies (SE1). This proves (iv). Furthermore,
(iv)$\implies$(iii) is trivial.

As for (iii)$\implies$(i) assume that (SE1) holds in $\mathcal K$.  The first task is to show that the composite
$X\circ Y$ for $X,Y\in {\mathcal K}$ can be obtained as a conformal composite (supremum) of $X$ with members $Y^{(f)}$
of ${\mathcal K}$, one for each $f\in E\setminus S(X,Y).$ Given such a position $f$, start an iteration with
$Z^{(\varnothing,f)}:=Y,$ and as long as $A\ne S(X,Y),$ apply (SE1) to $X,Z^{(A,f)}\in {\mathcal K}$, which then returns a sign vector
$$Z^{(A\cup \{ e\},f)}\in {\mathcal W}_A(X,Z^{(A,f)})\cap {\mathcal K}\subseteq {\mathcal W}_A(X,Y)\cap {\mathcal K}\mbox{ with }$$
$$Z^{(A\cup \{ e\},f)}_e=0 \mbox{ and } Z^{(A\cup \{ e\},f)}_f=(X\circ Z^{(A,f)})_f=(X\circ Y)_f.$$

In particular, $Z^{(A\cup \{ e\},f)}\in {\mathcal W}_{A\cup\{e\}}(X,Y)\cap {\mathcal K}$.
Eventually, the iteration stops with
\begin{align*}
&Y^{(f)}:=Z^{(S(X,Y),f)}\in {\mathcal W}_{\infty}(X,Y)\cap {\mathcal K} \mbox{ satisfying }\\
&Y^{(f)}_f=(X\circ Y)_f \mbox{ and } (X\circ Y^{(f)})_e=X_e \mbox{ for all } e \mbox{ separating } X  \mbox{ and } Y.
\end{align*}
Now take the supremum of $X$ and all $Y^{(f)}:$ then
$$X\circ Y= X\odot\bigodot_{f\in E\setminus S(X,Y)} Y^{(f)}$$
constitutes the desired representation.

Next consider a composition $X\circ X^{(1)}\circ\ldots\circ X^{(n)}$ of $n+1\ge 3$ sign vectors from $\mathcal K$.
By induction on $n$ we may assume that
$$X^{(1)}\circ\ldots \circ X^{(n)}=Y^{(1)}\odot\ldots \odot Y^{(m)}$$
where $Y^{(i)}\in {\mathcal K}$ for all $i=1,\ldots, m$. Since any supremum in $\{ \pm 1,0\}^E$ needs at most
$\# E$ constituents, we may well
choose $m=\# E$. Similarly, as the case $n=1$ has been dealt with, each $X\circ Y^{(i)}$ admits a commutative representation
$$X\circ Y^{(i)}=X\odot Z^{(m(i-1)+1)}\odot Z^{(m(i-1)+2)}\odot\ldots\odot Z^{(mi)} ~~~~(i=1,\ldots,m).$$
We claim that $Z^{(j)}$ and $Z^{(k)}$ commute for all $j,k\in \{ 1,\ldots, m^2\}$. Indeed,
$$Z^{(j)}\le X\circ Y^{(h)} \mbox{ and } Z^{(k)}\le X\circ Y^{(i)} \mbox{ for some } h,i\in \{ 1,\ldots, m\}.$$
Then $$Z^{(j)},Z^{(k)}\le (X\circ Y^{(h)})\circ Y^{(i)}=(X\circ Y^{(i)})\circ Y^{(h)}$$
because $Y^{(h)}$ and $Y^{(i)}$ commute, whence $Z^{(j)}$ and $Z^{(k)}$ commute as well. Therefore
\begin{align*}
X\circ X^{(1)}\circ\ldots\circ X^{(n)}&=X\circ Y^{(1)}\circ\ldots\circ Y^{(m)}= (X\circ Y^{(1)})\circ\ldots\circ (X\circ Y^{(m)})\\
&=X\odot Z^{(1)}\odot\ldots\odot Z^{(m^2)}
\end{align*}
gives the required representation. We conclude that $(E,\covectors)$ satisfies (C).

To establish (SE1) for $\covectors$, let  $X=X^{(1)}\odot\cdots \odot X^{(n)}$ and  $Y=Y^{(1)}\odot\cdots \odot Y^{(m)}$
with  $X^{(i)},Y^{(j)}\in {\mathcal K}$ for all $i,j$.
Take $e,f\in E$ such that $e$ separates $X$ and $Y$ and $f$ does not. We may assume that  $X^{(i)}_e=X_e$
for $1\le i\le h$,  $Y^{(j)}_e=Y_e$  for $1\le j\le k$, and  equal to zero otherwise
(where  $h,k\ge 1$). Since $\mathcal K$ satisfies (SE1) there exists $Z^{(i,j)}\in {\mathcal W}(X^{(i)},Y^{(j)})\cap {\mathcal K}$
such that $Z^{(i,j)}_e=0$ and $Z^{(i,j)}_f=(X^{(i)}\circ Y^{(j)})_f$ for $i\leq h$ and $j\leq k$.  Then the composition of
all $Z^{(i,1)}$ for $i\le h$, $X^{(i)}$ for $i>h$ and all $Z^{(1,j)}$ for $j\le k$, $Y^{(j)}$ for $j>k$ yields the
required sign vector $Z\in {\mathcal W}(X,Y)$ with $Z_e=0$ and $Z_f=(X\circ Y)_f$. We conclude that $(E,\covectors)$ is
indeed a strong elimination system by Lemma~\ref{lem:auxilliary}.
\end{proof}

From the preceding proof of (iv)$\implies$(i) we infer that the (supremum-)irreducibles of a strong elimination system
$(E,\covectors)$  are a fortiori irreducible with respect to arbitrary composition.

\section{Cocircuits of COMs}\label{sec:cocircuit}
In an OM, a cocircuit is a support-minimal non-zero covector, and the cocircuits form the unique minimal generating
system for the entire set of covectors provided that composition over an empty index set is allowed. Thus, in our context
the zero vector would have to be added to the generating set, i.e., we would regard it as a cocircuit as well.
The cocircuits of COMs that we will consider next should on the one hand generate the entire system and on the other
hand their restriction to any maximal face should be the set of cocircuits of the oriented matroid corresponding to
that face via Lemma~\ref{lem:fiber}.

For any $\mathcal{K}$ with $\mathcal{J}=\mathcal{J}(\covectors)\subseteq\mathcal{K}\subseteq\covectors$ denote by
$\Min(\mathcal{K})$ the set of all minimal sign vectors of $\mathcal{K}$. Clearly, $\Min(\bigodot\mathcal{K})=\Min(\mathcal{K})=\Min(\mathcal{J})$.
We say that $Y$ \emph{covers} $X$ in $\covectors=\bigodot\mathcal{J}$ (in symbols: $X \prec Y$) if $X < Y$ holds and there is
no sign vector $Z \in \covectors$ with $X < Z < Y$. The following set $\mathcal{C}$ is intermediate between $\mathcal{J}$ and $\covectors$:
$$\mathcal{C}=\mathcal{C}(\covectors):=\mathcal{J}(\covectors)\cup\{X\in\covectors: W\prec X\text{ for some }W\in\Min(\covectors)\}.$$
Since $\Min(\covectors)=\Min(\mathcal{J})$ and every cover $X\notin\mathcal{J}$ of some $W\in\Min(\mathcal{J})$ is above some other
$V\in\Min(\mathcal{J})$, we obtain:
$$\mathcal{C}=\mathcal{C}(\mathcal{J}):=\mathcal{J}\cup\{W\odot V: V,W\in \Min(\mathcal{J})\text{ and } W\prec W\odot V\}.$$

We will make use of the following variant of face symmetry restricted to comparable covectors:
\begin{itemize}
\item[{\bf (FS$^\le$)}] $X\circ -Y\in \covectors$ for all $X\le Y$ in $\covectors$.
\end{itemize}

Note that (FS) and (FS$^\le$) are equivalent in any system $\covectors$ satisfying (C), as one can let $X\circ Y$
substitute $Y$ in (FS). We can further weaken face symmetry by restricting it to particular covering pairs $X\prec Y$:

\begin{itemize}
\item[{\bf (FS$^{\prec}$)}] $W\circ -Y\in {\covectors}$ for all $W\in \Min({\covectors})$ and $Y\in\covectors$
with $W\prec Y$ in $\covectors$, or equivalently,
\item[] $W\circ -Y\in {\mathcal{C}}$ for all $W\in \Min(\mathcal{C})$ and $Y\in\mathcal{C}$ with $W\prec Y$.
\end{itemize}

Indeed, since sign reversal constitutes an order automorphism of $\{\pm1, 0\}^E$, we readily infer that in
(FS$^{\prec}$) $W\circ -Y$ covers $W$, for if there was $X\in\covectors$ with $W\prec X<W\circ -Y$, then
$W<W\circ -X<W\circ-(W\circ -Y)=W\circ Y=Y$, a contradiction.

To show that (FS$^{\prec}$) implies (FS$^\le$) takes a little argument, as we will see next.

\begin{proposition}\label{prop:cc_com}
Let $(E,\mathcal{J})$ be a system of sign vectors.
Then $(E, \bigodot\mathcal{J})$ is a COM such that $\mathcal{J}$ is its set of irreducibles
if and only if $\mathcal{C}=\mathcal{C}(\mathcal{J})$ satisfies {\normalfont (FS$^{\prec}$)}
and $\mathcal{J}$ satisfies {\normalfont (SE1)} and
\begin{itemize}
\item[{\bf (IRR)}] if $X=\bigodot_{i=1}^nX_i$ for $X,X_1, \ldots, X_n\in\mathcal{J} (n\geq 2)$,
then $X=X_i$ for some $1\leq i\leq n$.
\end{itemize}
\end{proposition}

\begin{proof}
First, assume that $(E,\covectors=\bigodot\mathcal{J})$ is a COM with $\mathcal{J}=\mathcal{J}(\covectors)$.
From Theorem~\ref{thm:cc_ses} we know that $\mathcal{J}$ satisfies (SE1), while (IRR) just expresses irreducibility.
Since $\covectors$ is the set of covectors of a COM, from the discussion preceding the theorem it follows that
$\covectors$ satisfies (FS$^{\prec}$). Consequently,  $\mathcal{C}=\mathcal{C}(\mathcal{J})$ satisfies (FS$^{\prec}$).

Conversely, in the light of Theorem~\ref{thm:cc_ses}, it remains to prove that (FS$^{{\prec}}$) for
$\mathcal C$ implies (FS$^{{\le}}$) for  $(E,\covectors)$. Note that for $W<X<Y$ in $\covectors$ we have $X\circ -Y=X\circ W\circ -Y$,
whence for $W<Y\in \covectors$ we only need to show $W\circ -Y\in \covectors$ when $W$ is a minimal sign vector of $\covectors$
(and thus belonging to ${\mathcal J}\subseteq {\mathcal C}$). Now suppose that $W\circ -Y\notin \covectors$ for
some covector $Y$ 
such that $\#Y^0$ is as large as possible. Thus as $Y\notin {\mathcal C}$ there exists
$X\in \covectors$ with $W\prec X<Y$. By (FS$^{\prec}$), $W\circ -X\in \covectors$ holds. Pick any element
$e\in S(W\circ -X\circ Y,Y)=W^0\cap\underline{X}$ and choose some $Z\in \covectors$ with $Z_e=0$ and $Z_f=(W\circ -X\circ Y)_f$
for all $f\in E\setminus S(W\circ -X\circ Y,Y)$
by virtue of (SE). In particular, $Y=X\circ Z$. Then necessarily $W<Z$ and $Y^0\cup \{ e\}\subseteq Z^0$, so
that $W\circ -Z\in \covectors$ by the maximality hypothesis. Therefore with Theorem~\ref{thm:cc_ses} we get
$$W\circ -Y=W\circ -(X\circ Z)=(W\circ -X)\circ(W\circ -Z)\in \covectors,$$
which is  a contradiction. This establishes (FS$^{\le}$) for $\covectors$ and thus completes the
proof of Proposition~\ref{prop:cc_com}.
\end{proof}

Proposition~\ref{prop:cc_com} yields the following alternative axiomatization of COMs in terms of covectors,
that is of independent interest:

\begin{corollary}\label{cor:covectors} A system $(E,\covectors)$ of sign vectors is a COM if and
only if $(E,\covectors)$ satisfies {\normalfont (CC)}, {\normalfont (SE1)} and {\normalfont (FS$^{\prec}$)}.
\end{corollary}

Let us now advance towards the axiomatization of COMs in terms of cocircuits.
Given a COM $(E,\covectors)$, we call the minimal sign vectors of $\covectors$ the {\it improper cocircuits} of $(E,\covectors)$.
A {\it proper cocircuit} is any sign vector $Y\in\covectors$ which {\it covers} some improper cocircuit $X$. {\it Cocircuit} then
refers to either kind, improper or proper. Hence, $\mathcal{C}(\covectors)$ is the set of all cocircuits of $\covectors$.
Note that in oriented matroids the zero vector is the only improper cocircuit and the usual OM cocircuits
are the proper cocircuits in our terminology. In lopsided systems $(E,\covectors)$, the improper cocircuits are the barycenters of maximal
hypercubes~\cite{bachdrko-12}. In a COM improper circuits are irreducible, but not all proper circuits need to be irreducible.
Here is the main result of this section.

\begin{theorem}\label{thm:coc}
Let $(E,\mathcal{C})$ be a system of sign vectors and let $\covectors:= \bigodot\mathcal{C}$. Then $(E,\covectors)$ is a COM
such that $\mathcal{C}$ is its set of cocircuits if and only if $\mathcal{C}$ satisfies {\normalfont (SE1)},{\normalfont (FS$^{\prec}$)}, and
\medskip\begin{itemize}
\item[{\bf (COC)}] $\mathcal{C}=\Min(\mathcal{C})\cup\{Y\in\bigodot\mathcal{C}: W\prec Y\text{ for some }W\in\Min(\mathcal{C})\}.$
\end{itemize}
\end{theorem}

\begin{proof} Let $(E,\covectors)$ be a COM and $\mathcal{C}$ be its set of cocircuits.  By Proposition~\ref{prop:cc_com},
$\mathcal{C}$ satisfies (FS$^{\prec}$). From the proof of Theorem \ref{thm:cc_ses}, part (ii)$\Rightarrow$(iv), we know that a sign vector
$Z$ demanded in (SE1) could always be chosen from the irreducibles, which are particular cocircuits. Therefore
${\mathcal C}= {\mathcal C}(L)$ satisfies (SE1). Finally, (COC) just expresses that $\mathcal{C}$ exactly comprises
the cocircuits of the set $\covectors$ it generates.

Conversely, $\covectors$ satisfies (CC) by definition. Since $\mathcal{J}(\covectors)\subseteq\mathcal{C}$ and
$\mathcal{C}$ satisfies (SE1), applying Theorem~\ref{thm:cc_ses} we conclude that $\mathcal{J}$ satisfies (SE1).
Consequently, as $\mathcal{C}$ satisfies (FS$^{\prec}$) and $\mathcal{J}$ satisfies (SE1), $\covectors$ is a COM
by virtue of Proposition~\ref{prop:cc_com}.
\end{proof}

To give a simple class of planar examples, consider the hexagonal grid, realized as the 1-skeleton of the regular tiling of
the plane with (unit) hexagons. A {\it benzenoid graph} is the $2$-connected graph formed by the vertices and edges from
hexagons lying entirely within the region bounded by some cycle in the hexagonal grid; the given cycle thus constitutes
the boundary of the resulting benzenoid graph~\cite{ham-11}. A \emph{cut segment} is any minimal (closed) segment of a
line perpendicular to some edge and passing through its midpoint such that the removal of all edges cut by the segment
results in exactly two connected components, one signed $+$ and the other $-$.
The ground set $E$ comprises all these cut segments. The set $\covectors$ then consists of all sign
vectors corresponding to the vertices and the barycenters (midpoints) of edges and 6-cycles (hexagons)
of this benzenoid graph. For verifying that $(E,\covectors)$ actually constitutes a COM, it is instructive
to apply Proposition~\ref{prop:cc_com}: the set $\mathcal J$ of irreducible members of $\covectors$ encompasses
the barycenter vectors of the boundary edges and of all hexagons of the benzenoid. The barycenter vectors of two
hexagons/edges/vertices are sign consistent exactly when they are incident. Therefore $\mathcal J$ generates
all covectors of  $\covectors$  via (CC). Condition (FS$^\prec$) is realized through inversion of an edge
at the center of a hexagon it is incident with. Condition (SE1) is easily checked by considering two cases
each (depending on whether $Z$ is eventually obtained as a barycenter vector of a hexagon or of an edge) for
pairs $X,Y$ of hexagon/edge barycenters.

\section{Hyperplanes, carriers, and halfspaces}\label{sec:hyperplanes}
For a system $(E,\covectors)$ of sign vectors, a {\it hyperplane} of $\covectors$ is the set
$$\covectors_e^0:=\{ X\in \covectors: X_e=0\} \mbox{ for some } e\in E.$$
The {\it carrier} $N(\covectors_e^0)$ of the hyperplane $\covectors_e^0$ is the union of all faces $F(X')$ of
$\covectors$ with $X'\in \covectors_e^0,$ that is,
\begin{align*}
N(\covectors_e^0):=\{ X\in \covectors: W\leq X \mbox{ for some } W\in \covectors_e^0\}.
\end{align*}
The {\it positive and negative (``open'') halfspaces} supported by the hyperplane $\covectors_e^0$ are
$$\covectors_e^+:= \{ X\in \covectors: X_e=+1\},$$
$$\covectors_e^-:=\{ X\in \covectors: X_e=-1\}.$$
The carrier $N(\covectors_e^0)$ minus $\covectors_e^0$ splits into its positive and negative parts:
$$N^+(\covectors_e^0):=\covectors^+_e\cap N(\covectors_e^0),$$
$$N^-(\covectors_e^0):=\covectors_e^-\cap N(\covectors_e^0).$$
The closure of the disjoint halfspaces $\covectors^+_e$ and $\covectors_e^-$ just adds the corresponding carrier:
$$\overline{\covectors_e^+}:=\covectors_e^+\cup N(\covectors_e^0)=\covectors_e^+\cup\covectors_e^0\cup N^-(\covectors_e^0),$$
$$\overline{\covectors_e^-}:=\covectors_e^-\cup N(\covectors_e^0)=\covectors_e^-\cup\covectors_e^0\cup N^+(\covectors_e^0).$$
The former is called the {\it closed positive halfspace} supported by $\covectors_e^0,$ and the latter is the corresponding
{\it closed negative halfspace}. Both overlap exactly in the carrier.

\begin{figure}[ht]
\includegraphics[width = .5\textwidth]{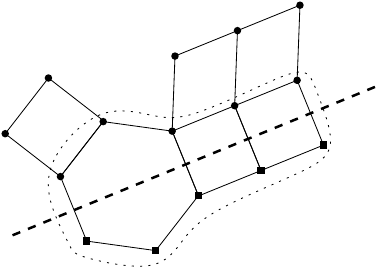}
\caption{A hyperplane (dashed), its associated open halfspaces (square and round vertices, respectively)
and the associated carrier (dotted) in a COM.}
\label{fig:Fig3}
\end{figure}

\begin{proposition} \label{carriers} Let $(E,\covectors)$ be a system of sign vectors. Then all its
hyperplanes, its carriers and their positive and negative parts, its halfspaces and their closures are
strong elimination systems or COMs, respectively, whenever $(E,\covectors)$ is such. If $(E,\covectors)$
is an OM, then so are all its hyperplanes and carriers.
\end{proposition}

\begin{proof} We already know that fibers preserve (C),(FS), and (SE). Moreover, intersections preserve the two properties (C) and (FS). Since $X'\le X$ and
$Y'\le Y$ imply both $X'\le X\circ Y, X'\le X\circ (-Y)$ and $Y'\le Y\circ X, Y'\le Y\circ (-X)$, we infer that (C) and (FS) carry over from
$\covectors$ to $N({\covectors}^0_e)$.

In what follows let $(E,\covectors)$ be a strong elimination system. We need to show that $N({\covectors}^0_e)$ satisfies (SE).
Let $X',Y'\in {\covectors}^0_e$ and $X,Y\in \covectors$ such that
$X'\le X$ and $Y'\le Y$. Then $S(X',Y')\subseteq S(X,Y)$. Apply (SE) to the pair $X,Y$ in $\covectors$ relative to some element $e'$ separating $X$ and $Y$.
This yields some $Z\in {\mathcal W}(X,Y)$ with $Z_{e'}=0$ and $Z_f=(X\circ Y)_f$ for all $f\in E\setminus S(X,Y).$ If $e'\in S(X',Y')$
as well, then apply (SE) to  $X',Y'$ in ${\covectors}_e^0$ giving $Z'\in {\mathcal W}(X',Y')\cap {\covectors}_e^0$ with $Z'_{e'}=0$
and $Z'_f=(X'\circ Y')_f$ for all $f\in E\setminus S(X,Y)\subseteq E\setminus S(X',Y').$ If $e'\in \underline{X}'\setminus \underline{Y}'$,
then put $Z':=Y'$. Else, if $e'\in E\setminus \underline{X}'$, put $Z':=X'$. Observe that all cases are covered as
$S(X',Y')=S(X,Y)\cap \underline{X}'\cap \underline{Y}'.$ We claim that in any case $Z'\circ Z$ is the required sign vector
fulfilling (SE) for $X,Y$ relative to $e'$. Indeed, $Z'\circ Z$ belongs to $N({\covectors}^0_e)$ since
$Z'\in {\covectors}^0_e$ and $Z\in \covectors$. Then $Z'\circ Z\in {\mathcal W}(X,Y)$ because ${\mathcal W}(X',Y')\subseteq {\mathcal W}(X,Y)$
and ${\mathcal W}(X,Y)$ is closed under composition. Let $f\in E\setminus S(X,Y)$. Then $X'_f,X_f,Y'_f,Y_f$ all commute.

In particular,
$$(Z'\circ Z)_f=Z'_f\circ Z_f=X'_f\circ Y'_f\circ X_f\circ Y_f=X'_f\circ X_f\circ Y'_f\circ Y_f=(X\circ Y)_f$$
whenever both $Y'_{e'}=Y_{e'}$ and $X'_{e'}=X_{e'}$ hold. If however $Y'_e=0,$ then $(Y'\circ Z)_f=Y'_f\circ X_f\circ Y_f=X_f\circ Y'_{f}\circ Y_f=(X\circ Y)_f$.
Else, if $X'_e=0,$ then $(X'\circ Z)_f=X'_f\circ X_f\circ Y_f=(X\circ Y)_f.$ This finally shows that the carrier of ${\covectors}^0_e$ satisfies (SE).

To prove that $\overline{\covectors_e^+}$ satisfies (SE) for a pair $X,Y\in \overline{\covectors_e^+}$ relative to some $e'\in S(X,Y),$
assume that $X\in \covectors^+_e$ and $Y\in N(\covectors^0_e)\setminus \covectors^+_e$
since (SE) has already been established for both $\covectors^+_e$ and $N(\covectors^0_e)$ and the required sign vector would equally
serve the pair $Y,X$. Now pick any $Y'\in \covectors^0_e$ with $Y'\le Y$. Then two cases can occur for $e'\in S(X,Y)$.

\medskip
{\bf Case 1.} $Y'_{e'}=0.$

\medskip
Then $(Y'\circ X)_{e'}=X_{e'}, S(Y'\circ X,Y)\subseteq S(X,Y)$, and $Y'\le Y'\circ X$, whence $Y'\circ X\in N(\covectors^0_e)$.
Applying (SE) to $Y'\circ X, Y$ in $N(\covectors^0_e)$ relative to $e'$ yields $Z'\in {\mathcal W}(Y'\circ X,Y)\subseteq {\mathcal W}(X,Y)$ with $Z'_{e'}=0$ and
$$Z'_f=Y'_f\circ X_f\circ Y_f=X_f\circ Y'_f\circ Y_f=(X\circ Y)_f \mbox{ for all } f\in E\setminus S(X,Y).$$

\medskip
{\bf Case 2.} $Y'_{e'}=Y_{e'}.$

\medskip
As above we can select $Z\in {\mathcal W}(X,Y)$ with $Z_{e'}=0$ and $Z_f=(X\circ Y)_f$ for all $f\in E\setminus S(X,Y)$.
Analogously choose $Z'\in {\mathcal W}(X,Y')$ with $Z'_{e'}=0$ and $Z'_f=(X\circ Y')_f$ for all $f\in E\setminus S(X,Y').$
We claim that in this case $Z'\circ Z$ is a sign vector from $\covectors^+_e$ as required for $X,Y$ relative to $e'$.
Indeed, $Z'_e=(X\circ Y')_e=+1=(X\circ Y)_e$ because $X_e=+1,Y'_e=0$ and consequently $e\notin S(X,Y').$ For $f\in E\setminus S(X,Y)$
we have
$$(Z'\circ Z)_f=X_f\circ Y'_f\circ X_f\circ Y_f=X_f\circ X_f\circ Y'_f\circ Y_f=(X\circ Y)_f$$
by commutativity, similarly as above. This proves that  $\overline{\covectors_e^+}$ satisfies (SE).

To show that $N^+(\covectors^0_e)=\covectors^+_e\cap N(\covectors^0_e)$ satisfies (SE), we can apply (SE) to some pair
$X,Y\in N^+(\covectors^0_e)$ relative to some $e'\in S(X,Y)$ first within $N(\covectors^0_e)$ and then within $\covectors^+_e$
to obtain two sign vectors $Z'\in N(\covectors^0_e)\cap {\mathcal W}(X,Y)$ and $Z\in {\covectors}^+_e\cap {\mathcal W}(X,Y)$
such that $Z'_{e'}=0=Z_{e'}$ and $Z'_{f}=(X\circ Y)_f=Z_f$ for all $f\in E\setminus S(X,Y).$ Then $Z'\le Z'\circ Z\in N(\covectors^0_e)$
and $(Z'\circ Z)_e=(X\circ Y)_e=+1$ as $e\notin S(X,Y).$ Moreover, $(Z'\circ Z)_f=(X\circ Y)_f$ for all $f\in E\setminus S(X,Y).$
This establishes (SE) for $N^+(\covectors^0_e).$ The proofs for $\overline{\covectors_e^-}$ and $N^-(\covectors^0_e)$ are
completely analogous. The last statement of the proposition is then trivially true because the zero vector, once present in
$\covectors$, is also contained in all hyperplanes (and hence the carriers).
\end{proof}

A particular class of COMs obtained by the above proposition are halfspaces of OMs. These are usually called \emph{affine oriented matroids},
see~\cite{Karlander-92} and~\cite[p. 154]{bjvestwhzi-93}. Karlander~\cite{Karlander-92} has shown how an OM can be reconstructed
from any of its halfspaces. The proof of his intriguing axiomatization of affine oriented matroids, however, has a gap, which
has been filled only recently~\cite{Zhu-15}. Only few results exist about the complex given by an affine oriented
matroid~\cite{Dong-06,Dong-08}.

We continue with the following recursive characterization of COMs:

\begin{theorem}\label{hyperplane}  Let $(E,\covectors)$ be a semisimple system of sign vectors. Then
 $(E,\covectors)$ is a strong elimination system if and only if the following four requirements are met:
\begin{itemize}
\item[(1)] the composition rule (C) holds in $(E,\covectors)$,
\item[(2)] all hyperplanes of $(E,\covectors)$ are strong elimination systems,
\item[(3)] the tope graph of $(E,\covectors)$ is a partial cube,
\item[(4)] for each pair $X,Y$ of adjacent topes (i.e., with $\#S(X,Y)=1$) the barycenter of the corresponding edge,
i.e. the sign vector $\frac{1}{2}(X+Y),$ belongs to $\covectors$.
\end{itemize}

Moreover, $(E,\covectors)$ is a COM if and only if it satisfies (1),(3),(4), and
\begin{itemize}
\item[(2$'$)] all hyperplanes of $(E,\covectors)$ are COMs,
\end{itemize}

In particular, $(E,\covectors)$ is an OM if and only if it satisfies (1),(4), and
\begin{itemize}
\item[(2$''$)] all hyperplanes of $(E,\covectors)$ are OMs,
\item[(3$'$)] the tope set of $(E,\covectors)$ is a \emph{simple acycloid, see~\cite{ha-93}}, i.e., induces a partial cube and satisfies \normalfont{(Sym)}.
\end{itemize}
\end{theorem}

\begin{proof} The ``if'' directions of all three assertions directly follow from Propositions~\ref{carriers} and~\ref{OMCtoTopeGraph}.
Conversely, using (1) and Lemma~\ref{lem:auxilliary}, we only need to verify (SE$^{=}$) to prove the first assertion.

To establish (SE$^{=}$), let $X$ and $Y$ be any different sign vectors from $\covectors$. Assume that $\underline{X}=\underline{Y}$ and
$e\in S(X,Y)$. If the supports are not all of $E$, then we can apply (SE$^{=}$) to the hyperplane associated with a zero coordinate
of $X$ and $Y$ according to condition (2) and obtain a sign vector $Z$ as required. Otherwise, both $X$ and $Y$ are topes.
Then a shortest path joining $X$ and $Y$ in the tope graph is indexed by the elements of $S(X,Y)$ and thus includes an edge
associated with $e$. Then the corresponding barycenter map $Z$ (that belongs to $\covectors$ by condition (4)) of this edge
does the job. Thus $(E,\covectors)$ is a semisimple strong elimination system.

In order to complete the proof of the second assertion it remains to establish (FS$^{\le}$). So let $X$ and $Y$ be any
different sign vectors from $\covectors$ with $X\circ Y=Y$. In particular, $X$ is not a tope and $Y$ belongs to the face $F(X)$.
If the support $\underline{Y}$ does not equal $E$, then again we find a common zero coordinate of $X$ and $Y$,
so that we can apply (FS$^{\le}$) in the corresponding hyperplane to yield the sign vector opposite to $Y$ relative to $X$.
So we may assume that $Y$ is a tope. Since $(E,\covectors)$ is a semisimple strong elimination system,
from Proposition~\ref{OMCtoTopeGraph} we infer that the tope graph of $F(X)$ is a partial cube containing at least two topes.
Thus there exists a tope $U\in F(X)$ adjacent to $Y$ in the tope graph,  say $S(U,Y)=\{ e\}$. Let $W$ be the
barycenter map of this edge. Applying (FS$^{\le}$) for the pair $X,W$ in the hyperplane $\covectors_e^0$ relative to $e$
we obtain $X\circ(-W)\in\covectors_e^0$. By (1) we have $X\circ(-W)\circ U\in\covectors$. Since $X\circ(-W)\circ U=X\circ(-Y)$
this concludes the proof.

As for the third assertion, note that symmetric COMs are OMs and symmetry for non-topes is implied by symmetry for hyperplanes.
\end{proof}

\section{Decomposition and amalgamation}\label{sec:amalgam}

Proposition~\ref{carriers} provides the necessary ingredients for a decomposition of a COM, which is not an OM,
into smaller COM constituents. Assume that $(E,\covectors)$ is a semisimple COM that is not an OM.
Put $\covectors':=\covectors^-_e$ and $\covectors'':=\overline{\covectors^+_e}$. Then
$\covectors=\covectors'\cup \covectors''$ and $\covectors'\cap \covectors''=N^-(\covectors^0_e)$.
Since $X$ determines a maximal face not included in $\covectors^0_e$,
we infer that $\covectors'\setminus \covectors''\ne \varnothing$ and trivially
$\covectors''\setminus \covectors'\ne \varnothing$. By Proposition~\ref{carriers},
all three systems $(E,\covectors'),$ $(E,\covectors'')$, and $(E,\covectors'\cap \covectors'')$
are COMs, which are easily seen to be semisimple.

Moreover, $\covectors'\circ \covectors''\subseteq \covectors'$ holds trivially. If $W\in \covectors^0_e$
and $X\in \covectors^-_e,$ then $W\circ X\in F(W)\subseteq N(\covectors^0_e)$, whence
$\covectors''\circ \covectors'\subseteq \covectors''$. This motivates the following amalgamation process
which in a way reverses this decomposition procedure.

We say that a system $(E,\covectors)$ of sign vectors is a {\it COM amalgam} of two semisimple COMs $(E,\covectors')$
and $(E,\covectors'')$ if the following conditions are satisfied:

\begin{itemize}
\item[(1)] $\covectors=\covectors'\cup \covectors''$ with $\covectors'\setminus \covectors'', \covectors''\setminus \covectors', \covectors'\cap \covectors''\ne \varnothing$;

\item[(2)] $(E,\covectors'\cap \covectors'')$ is a semisimple COM;

\item[(3)] $\covectors'\circ \covectors''\subseteq \covectors'$ and $\covectors''\circ \covectors'\subseteq \covectors''$;

\item[(4)] for $X\in \covectors'\setminus \covectors''$ and $Y\in \covectors''\setminus \covectors'$ with $X^0=Y^0$ there
exists a shortest path in the graphical hypercube on $\{ \pm 1\}^{E\setminus X^0}$ for which all its vertices and barycenters
of its edges belong to $\covectors\setminus X^0$.
\end{itemize}

\begin{proposition} \label{amalgam} The COM amalgam of semisimple COMs $(E,\covectors')$ and $(E,\covectors'')$
constitutes a semisimple COM $(E,\covectors)$ for which every maximal face is a maximal face of at least one
of the two constituents.
\end{proposition}

\begin{proof} $\covectors=\covectors'\cup \covectors''$ satisfies (C) because $\covectors'$ and $\covectors''$ do and for
$X\in \covectors'$ and $Y\in \covectors''$ one obtains $X\circ Y\in \covectors'\subseteq \covectors$ and $Y\circ X\in \covectors''\subseteq \covectors$ by (3).
Then $\covectors$ also satisfies (FS$^{\le}$) since for $X\le Y=X\circ Y$ in $\covectors$ the only nontrivial case is that $X\in \covectors'$
and $Y\in \covectors''$, say. Then $Y=X\circ Y\in \covectors'$ by (3), whence $X\circ -Y\in \covectors'\subseteq \covectors$.

Every minimal sign vector $X\in \covectors$, say $X\in \covectors',$ yields the face $F(X)=\{ X\circ Y: Y\in \covectors\}\subseteq \covectors'\circ \covectors\subseteq \covectors'$.
It is evident that $(E,\covectors)$ is semisimple.

By Lemma~\ref{lem:auxilliary}, it remains to show (SE$^{=}$) for two sign vectors $X$ and $Y$ of $\covectors$ with $X^0=Y^0$,
where $X\in \covectors'\setminus\covectors''$ and $Y\in \covectors''\setminus \covectors'$. Then let $e\in S(X,Y)$ and
$f\in E\setminus S(X,Y)$. Then the barycenter of an $e$-edge on a shortest path $P$ from $X\setminus X^0$ to $Y\setminus X^0$
between $\covectors'\setminus X^0$ and $\covectors''\setminus X^0$ (guaranteed by condition $(4)$) yields the desired sign
vector $Z\in \covectors$ with $Z_e=0$, $X^0\subseteq Z^{0},$ and $Z\in {\mathcal W}(X,Y)$. Since $X^0=Y^0$, we have $X_f=Y_f$ by
the choice of $f$. Since $P$ is shortest, we get $Z_f=(X\circ Y)_f$.
\end{proof}

Summarizing the previous discussion and results, we obtain
\begin{corollary} \label{amalgam-decomposition} Semisimple COMs are obtained via successive COM amalgamations from their
maximal faces (that can be contracted to OMs). 
\end{corollary}

\section{Euler-Poincar\'e formulae}\label{sec:Euler}

In this section, we generalize the Euler-Poincar\'e formula known for OMs to COMs, which involves the rank function.
This is an easy consequence of decomposition and amalgamation. In the case of lopsided systems and their hypercube cells
the rank of a cell is simply expressed as the cardinality of the zero set of its associated covector.

Given an OM of rank $r$, for $0\leq i\leq r-1$ one defines $f_i$ as the number of cells of dimension $f_i$ of the
corresponding decomposition of the $(r-1)$-sphere, see Section~\ref{sec:complexes} for more about this representation.
It is well-known (cf. \cite[Corollary 4.6.11]{bjvestwhzi-93}) that $\sum_{i=0}^{r-1}(-1)^if_i=1+(-1)^{r-1}$. Adding
the summand $(-1)^{-1}f_{-1}=-1$ here artificially yields $\sum_{i=-1}^{r-1}(-1)^if_i=(-1)^{r-1}$. Multiplying
this equation by $(-1)^{r-1}$ and substituting $i$ by $r-1-j$ yields
$$\sum_{j=0}^r(-1)^jf_{r-1-j}=\sum_{i=1}^{r-1}(-1)^{r-1-i}f_i=1.$$
As $f_{r-1-j}$ gives the number of OM faces of rank $j$ we can restate this formula in covector
notation as $\sum_{X\in\covectors}(-1)^{r(X)}=1$, where $r(X)$ is the rank of the OM $F(X)\setminus \underline{X}$.
We define the rank of the covector of a COM in the same way.

Since COMs arise from OMs by successive COM amalgamations, which do not create new faces, and at a
step from $\covectors'$ and $\covectors''$ to the amalgamated $\covectors$ each face in the intersection
is counted exactly twice, we obtain
$$\sum_{X\in\covectors}(-1)^{r(X)}=\sum_{X\in\covectors'}(-1)^{r(X)}+\sum_{X\in\covectors''}(-1)^{r(X)}-\sum_{X\in\covectors'\cap\covectors''}(-1)^{r(X)}=1.$$

\begin{proposition}\label{comeuler}
 Every COM $(E,\covectors)$ satisfies the Euler-Poincar\'e formula $\sum_{X\in\covectors}(-1)^{r(X)}=1$.
\end{proposition}

We now characterize lopsided systems in terms of an Euler-Poincar\'e formula.
A system $(E,\covectors)$ is said to satisfy the {\it Euler-Poincar\'e formula for zero sets} if
$$\sum_{X\in \covectors} (-1)^{\#X^0}=1.$$

\begin{proposition} \label{Euler}
The following assertions are equivalent for a system $(E,\covectors)$:
\begin{itemize}
\item[(i)] $(E,\covectors)$ is lopsided, that is, $(E,\covectors)$ is a COM satisfying (IC);
\item[(ii)]{\normalfont\cite{Wie}} every topal fiber of $(E,\covectors)$ satisfies the Euler formula  for zero sets, and $\covectors$ is determined
by the topes in the following way: for each sign vector $X\in \{ \pm 1,0\}^E,$ $X\in \covectors\Rightarrow X\circ Y\in \covectors$ for all $Y\in \{ \pm 1\}^E$;
\item[(iii)] every contraction of a topal fiber of $(E,\covectors)$ satisfies the Euler formula  for zero sets in its own right.
\end{itemize}
\end{proposition}

\begin{proof} Deletions, contractions, and fibers of lopsided sets are COMs satisfying (IC) as well, that is, are again lopsided.
In case of a lopsided system $(E,\covectors)$ for every $X\in\covectors$ we have $r(X)=\#X^0$. Therefore by Proposition~\ref{comeuler} $(E,\covectors)$ satisfies the Euler formula  for zero sets. This proves the implication (i)$\Rightarrow$(iii).

As for (iii)$\Rightarrow$(ii), we proceed by induction on $\#X^0$ for $X\in\covectors$. Assume that $X^0$ is not empty. Pick $e\in X^0$ and delete the coordinate subset $X^0\setminus e$ from $X$. Consider the topal fiber $\mathcal{R}=\{X'\in\covectors : X'\setminus \underline{X}=X\setminus \underline{X}\}$ relative to $X$ and $\underline{X}$, and contract $\mathcal{R}$ to $\mathcal{R}/(X^0\setminus e)$. Let $U^{(e)}$ denote the (unit) sign vector on $E$ with $U^{(e)}_e=+1$ and $U^{(e)}_f=0$ for $f\neq e$. Since $\mathcal{R}/(X^0\setminus e)$ satisfies the Euler-Poincar\'e formula for zero sets, both $X\circ U^{(e)}$ and $X\circ -U^{(e)}$ must belong to $\covectors$. By the induction hypothesis
$$(X\circ U^{(e)})\circ Z, (X\circ -U^{(e)})\circ Z\in\covectors \mbox{ for all } Z\in\{\pm 1\}^E,$$
whence indeed $X\circ Y\in \covectors$ for all $Y\in\{\pm 1\}^E$.

To prove the final implication (ii)$\Rightarrow$(i), we employ the recursive characterization of Theorem~\ref{hyperplane}. Since (IC) holds by the implication for $X\in\covectors$ in (ii),
property (1) of this theorem is trivially fulfilled. Observe that $\{\pm 1\}\subseteq\{X_e : X\in\covectors\}$ because the topal fiber relative to $X\in\covectors$ and $\underline{X}\neq E$ contains all
possible $-1,+1$ entries. If $X,Y$ are two topes of $\covectors$ with $S(X,Y)=\{e\}$, then the topal fiber relative to $X$ and $E\setminus e$ must contain $\frac{1}{2}(X+Y)$
by virtue of the Euler-Poincar\'e formula for zero sets. This establishes property (4) and (RN1).

Suppose that the topes of $\covectors$ do not form a partial cube in $\{\pm 1\}^E$. Then choose topes $X$ and $Y$ with $\#S(X,Y)\geq 2$ as small as possible such that the topal
fiber $\mathcal{R}$ relative to $X$ and $E\setminus S(X,Y)$ include no other topes than $X$ or $Y$. The formula for zero sets implies that this topal fiber $\mathcal{R}$ must contain
at least some $Z\in\covectors$ with $Z^0$ of odd cardinality. Then for $e\neq f$ in $S(X,Y)$ one can select signs for some tope $Z'$ conforming to $Z$ such that $Z'_eZ'_f\neq X_eX_f=Y_eY_f$.
Hence $\mathcal{R}$ contains the tope $Z'$ that is different from $X$ and $Y$, contrary to the hypothesis. This contradiction establishes that $\covectors$ fulfills property (3) and is semisimple.

Consider the hyperplane $\covectors_e^0$ and the corresponding halfspaces $\covectors_e^+$ and $\covectors_e^-$( which are two disjoint topal fibers of $\covectors$). Then the formula
$$\sum_{X\in\covectors}(-1)^{\#X^0}-\sum_{Y\in\covectors_e^+}(-1)^{\#Y^0}-\sum_{Z\in\covectors_e^-}(-1)^{\#Z^0}=-1$$
amounts to $$\sum_{W\in\covectors_e^0\setminus e}(-1)^{\#W^0}=1,$$
showing that the hyperplane after semisimplification satisfies the Euler-Poincar\'e formula. The analogous conclusion holds for any topal fiber $\covectors\setminus A$ of any 
$X\in\covectors$ with $A\subseteq\underline{X}$ because taking topal fibers and contractions commute. By induction we conclude that $(E,\covectors)$ is a COM satisfying (IC), 
that is, a lopsided system.
\end{proof}

Note that the equivalence of (i) and (ii) in Proposition~\ref{Euler} rephrases a result by Wiedemann~\cite{Wie} on lopsided sets. Observe that in condition (iii) 
one cannot dispense with contractions as the example $\covectors=\{+00\}$ shows. Neither can one weaken condition (ii) by dismissing topal fibers: consider a path 
in the 1-skeleton of $[-1,+1]^3$ connecting five vertices of the solid cube, which would yield an induced but non-isometric path of the corresponding graphical 
3-cube. Let $\covectors$ comprise the five vertices and the barycenters of the four edges, being represented by their sign vectors. Then all topal fibers except 
one satisfy the first statement in (ii), the second one being trivially satisfied.

\section{Ranking COMs}\label{sec:ranking}
Particular COMs naturally arise in order theory. For the entire section, let $(P,\leq)$ denote an ordered set (alias poset), that is, a finite set $P$ endowed 
with an order (relation) $\leq$.  A {\it ranking} (alias weak order) is an order for which incomparability is transitive. Equivalently, an order $\leq$ on $P$ 
is a ranking exactly when $P$ can be partitioned into antichains (where an \emph{antichain} is a set of mutually incomparable elements) $A_1, \ldots, A_k$, 
such that $x\in A_i$ is below $y\in A_j$ whenever $i<j$. An order $\leq$ on $P$ is  \emph{linear} if any two elements of $P$ are comparable, that is, all 
antichains are trivial (i.e., of size $<2$). An order $\leq'$ \emph{extends} an order $\leq$  on $P$ if $x\leq y$ implies $x\leq' y$ for all $x,y\in P$. 
Of particular interest are the linear extensions and, more generally, the ranking extensions of a given  order $\leq$ on $P$.

Let us now see how to associate a set of sign vectors to an order $\leq$ on $P = \{1,2,\ldots,n\}$. For this purpose take $E$ to be the set of all 2-subsets 
of $P$ and encode $\leq$ by its characteristic sign vector $X^{\leq} \in \{0,\pm 1\}^E$, which to each 2-subset $e = \{ i,j\}$ assigns $X^{\leq}_e = 0$ if 
$i$ and $j$ are incomparable, $X^{\leq}_e = +1$ if the order agrees with the natural order on the 2-subset $e$, and else $X^{\leq}_e = -1$. In the sign 
vector representation the different components are ordered with respect to the lexicographic natural order of the 2-subsets of $P$.

The composition of sign vectors from different orders $\leq$ and $\leq'$ does not necessarily return an order again. Take for instance, $X^{\leq}=+++$ 
coming from the natural order on $P$ and $X^{\leq'}=0-0$ coming from the order with the single (nontrivial) comparability $3\leq'1$. The composition 
$X^{\leq'}\circ X^{\leq}$ equals $+-+$, which signifies a directed 3-cycle and thus no order. The obstacle here is that $X^{\leq'}$ encodes an order 
for which one element is incomparable with a pair of comparable elements. Transitivity of the incomparability relation is therefore a necessary condition 
for obtaining a COM.

We denote by $\mathcal{R}(P,\leq)$ the simplification of the set of sign vectors associated to all ranking extensions of $(P,\leq)$. Note that the 
simplification amounts to omitting the pairs of the ground set corresponding to pairs of comparable pairs of $P$.

\begin{theorem}
 Let $(P,\leq)$ be an ordered set. Then $\mathcal{R}(P,\leq)$ is a realizable COM, called the \emph{ranking COM} of $(P,\leq)$.
\end{theorem}
\begin{proof}
The composition $X\circ Y$ of two sign vectors $X$ and $Y$ which encode rankings has an immediate order-theoretic interpretation: each (maximal) 
antichain of the order $\leq_X$ encoded by $X$ gets ordered according to the order $\leq_Y$ corresponding to $Y$. Similarly, in order to realize 
$X\circ -Y$ one only needs to reverse the order $\leq_Y$ before imposing it on the antichains of $\leq_X$. This establishes conditions (C) and (FS). 
To verify strong elimination (SE$^=$), assume that $X$ and $Y$ are given with $\underline{X}=\underline{Y}$, so that the corresponding rankings have 
the same antichains. These antichains may therefore be contracted (and at the end of the process get restored again). Now, for convenience we may 
assume that $X$ is the constant $+1$ vector, thus representing the natural linear order on $P$. Given $e = \{ i,j\}$ with $i <_X j$, let $Y_e = -1$, 
that is, $j <_Y i$. To construct a sign vector $Z$ with $Z_e = 0$ and $Z_f = X_f$ whenever $Y_f = X_f$, take the sign vector of the ranking corresponding 
to $X$ but place the subchain $\{h : i <_X h <_X j \mbox{ and } h <_Y j \}$ directly below the newly created antichain $\{ i,j\}$, and 
$\{h : i <_X h <_X j \mbox{ and } j <_Y h \}$ directly above it, while leaving everything else in the natural order. This establishes that $\mathcal{R}(P,\leq)$ 
is a COM. Realizability of $\mathcal{R}(P,\leq)$ will be confirmed in the third paragraph below. 
\end{proof}

To provide an example for a ranking COM and also illustrate the preceding construction, consider the ordered set (``fence'') shown in Figure~\ref{fig:poset}. 
In Figure~\ref{fig:lextgraph}, the sign vector $X$ encodes the natural order $<$ and $Y$ the ordering $3 <_Y 2 <_Y j = 5 <_Y i = 1 <_Y 4$, while $Z$ encodes the 
intermediate ranking with $2 <_Z 3 <_Z 1$ and $5 <_Z 4$.

\begin{figure}[htb]
   \centering
   \subfigure[\label{fig:poset}]{
    \includegraphics[width = .15\textwidth]{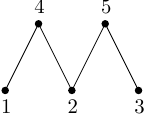}
   }
   \subfigure[\label{fig:lextgraph}]{
    \includegraphics[width = .35\textwidth]{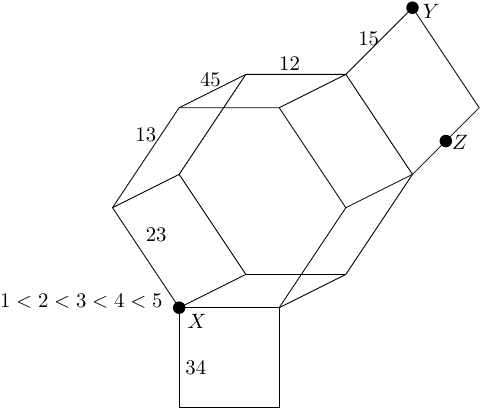}
   }
   \subfigure[\label{fig:weakexts}]{
    \includegraphics[width = .35\textwidth]{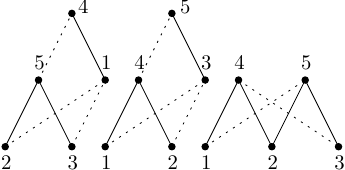}
   }
   \caption{From \subref{fig:poset} an ordered set $(P,\le)$ to \subref{fig:lextgraph} the ranking COM $\mathcal{R}(P,\le)$ comprising three maximal faces determined by \subref{fig:weakexts} the minimal rankings in $\mathcal{R}(P,\le)$.}
  \end{figure}

The ranking COM $\mathcal{R}(P,\le)$ is the natural host of all linear extensions of $(P,\le)$ (as its topes), where the interconnecting rankings signify the cell structure. 
The \emph{linear extension graph} of an ordered set $(P,\leq)$ is defined on the linear extensions of $(P,\leq)$, where two linear extensions are joined by an edge iff they 
differ on the order of exactly one pair of elements. Thus, the linear extension graph of $(P,\le)$ is the tope graph of $\mathcal{R}(P,\le)$.  A number of geometric and 
graph-theoretical features of linear extensions have been studied by various authors~\cite{Reu,Naa,Mas}, which can be expressed most naturally in the language of COMs.

\label{real} One such result translates to the fact that ranking COMs are realizable. To see this, first consider the \emph{braid arrangement} of type $B_n$, i.e., 
the central hyperplane arrangement $\{H_{ij}: 1\leq i<j\leq n\}$ in $\mathbb{R}^n$, where $H_{ij}=\{x\in \mathbb{R}^n: x_i=x_j\}$ and the position of any point 
in the corresponding halfspace $\{x\in \mathbb{R}^n: x_i<x_j\}$ is encoded by $+$ with respect to $H_{ij}$. The resulting OM is known as the \emph{permutahedron}~\cite{bjvestwhzi-93}. 
Given an order $\trianglelefteq$ on $P=\{1,\ldots, n\}$ consider the arrangement $E=\{H_{ij}: i, j\text{ incomparable}\}$ restricted to the open polyhedron  
$\bigcap_{i\vartriangleleft j}\{x\in \mathbb{R}^n: x_i<x_j\}$. The closure of the latter intersected with the unit cube $[0,1]^n$ coincides with the  
\emph{order polytope}~\cite{Sta-86} of $(P,\trianglelefteq)$. It is well-known that the maximal cells of the braid arrangement restricted to the order polytope 
of $(P,\trianglelefteq)$ correspond to the linear extensions of $(P,\trianglelefteq)$. Thus, the COM realized by the order polytope and the braid arrangement 
has the same set of topes as the ranking COM. By the results of Section~\ref{sec:topes}, this implies that both COMs coincide. In particular, ranking COMs are realizable.

We will now show how other notions for general COMs translate to ranking COMs.  A  face  $\mathcal F$ of $\mathcal R(P,\le)$, as defined in Section~\ref{sec:minors}, 
can be viewed as the set of all rankings that extend some ranking extension $\le'$ of $\le$. Hence ${\mathcal F}\cong\mathcal R(P,\le')$, i.e., all faces of a ranking 
COM are also ranking COMs. The minimal elements of $\mathcal{R}(P,\le)$ with respect to sign ordering (being the improper cocircuits of $\mathcal{R}(P,\le)$) are 
the minimal ranking extensions of $(P,\le)$, see Figure~\ref{fig:weakexts}.

A hyperplane $\mathcal{R}^0_e$ of ${\mathcal R}:=\mathcal{R}(P,\le)$ relative to $e=\{i,j\}\in E$ corresponds to those ranking extensions of $(P,\le)$ leaving $i,j$ 
incomparable. Thus,  $\mathcal{R}^0_e$ can be seen as the ranking COM of the ordered set obtained from $(P,\le)$ by identifying $i$ and $j$. The open halfspace 
$\mathcal{R}^+_e$ corresponds to those ranking extensions fixing the natural order on $i,j$ and is therefore the ranking COM of the ordered set $(P,\le)$ extended 
with the natural order on $i,j$. The analogous statement holds for $\mathcal{R}^-_e$. Similarly, the carrier of $\mathcal R$ relative to $e$ can be seen as the 
ranking COM of the ordered set arising as the intersection of all minimal rankings of $(P,\le)$ not fixing an order on $i,j$. So, in all three cases the resulting 
COMs are again ranking COMs.

One may wonder which are the ordered sets whose ranking COM is an OM or a lopsided system. The maximal cells in Figure~\ref{fig:lextgraph} are symmetric and 
therefore correspond to OMs. 

\begin{proposition}\label{prop:om}
The ranking COM of $(P,\leq)$ is an OM if and only if $\leq$ is a ranking. In this case, $\mathcal{R}(P,\leq)$ and its proper faces are products of permutohedra. 
\end{proposition}

\begin{proof}
Since any OM has a unique improper cocircuit, $\leq$ needs to be a ranking in order to have that $\mathcal{R}(P,\leq)$ is an OM. On the other hand, it is 
easy to see that if $\leq$ is a ranking on $P$, then $\mathcal{R}(P,\leq)$ is a product of permutohedra and, in particular, is symmetric, yielding the claim.
\end{proof}

\begin{proposition} The ranking COM of $(P,\leq)$ is a lopsided system if and only if $(P,\leq)$ has width at most 2. In this case, the tope graph of $\mathcal{R}(P,\leq)$ is the covering graph (i.e., undirected Hasse diagram) of a distributive lattice.
\end{proposition}
\begin{proof}
 If $(P,\leq)$ contains an antichain of size $3$, then the corresponding face of $\mathcal{R}(P,\leq)$ does not satisfy ideal composition, so that $\mathcal{R}(P,\leq)$ is not lopsided. Conversely, if all antichains have size at most $2$, then the zero entries of a sign vector $X$ encoding a ranking of $(P,\leq)$ correspond to maximal antichains of size $2$. Thus, choosing any sign on a zero entry just corresponds to fixing a linear order on the two elements of the antichain. Since these antichains are maximal and $X$ encodes a ranking, the resulting sign vector encodes a ranking, too. This proves ideal composition.

Let us now prove the second part of the claim. By Dilworth's Theorem~\cite{Dil-50}, $(P,\leq)$ can be covered by two disjoint chains, $C$ and $D$.
A linear extension of $(P,\leq)$ corresponds to an order-preserving mapping of $C$ to positions between consecutive elements of $D$ or above or below its maximal or minimal element, respectively. A linear extension $\trianglelefteq$ of $(P, \leq)$ can thus be codified by an order-preserving mapping $f$ from $C$ to the chain $\hat{D}=D\cup\hat{1}$, i.e., $D$ with a new top element $\hat{1}$ added: $f(c)=d\in\hat{D}$ signifies that the subchain $f^{-1}(d)$ of $C$ immediately precedes $d$ in the linear extension $\trianglelefteq$. If there are no comparabilities between $C$ and $D$ in $(P,\leq)$, then the tope graph of $\mathcal{R}(P,\leq)$ is the covering graph of the entire distributive lattice $L$ of order-preserving mappings from $C$ to $\hat{D}$ since covering pairs in $L$ correspond to pairs of linear extensions which are distinguished by a single neighbors swap. Additional covering relations between $C$ and $D$ yield lower and upper bounds for the order-preserving mappings, whence the resulting (distributive) linear-extension lattice of $(P,\leq)$ constitutes some order-interval of $L$.
\end{proof}

\begin{figure}[htb]
   \centering
   \subfigure[\label{fig:width2}]{
    \includegraphics[width = .15\textwidth]{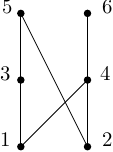}
   }
   \subfigure[\label{fig:distlatt}]{
    \includegraphics[width = .35\textwidth]{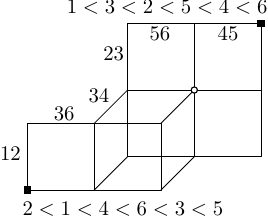}
   }
   \caption{From~\subref{fig:width2} the Hasse diagram of  $(P,\le)$ having width 2 to~\subref{fig:distlatt} the tope graph of the lopsided system 
   $\mathcal{R}(P,\le)$ oriented as a distributive lattice.}\label{fig:Fig4}
  \end{figure}

In Figure~\ref{fig:width2}, an ordered set $(P, \leq)$ of width $2$ is displayed, which has the natural order on $\{1, \ldots, 6\}$ among its linear 
extensions. Figure~\ref{fig:distlatt} shows the tope graph of the lopsided system $\mathcal{R}(P,\leq)$ and highlights the pair of diametrical 
vertices that determine the distributive lattice orientation (and its opposite); the natural order is associated with the (median) vertex 
(indicated by a small open circle). If we added the compatibility $3<6$ to the Hasse diagram, then the tope graph shrinks by collapsing the 
(two) edges corresponding to $\{3, 6\}$. The resulting graph with $F_7=13$ vertices is known as the ``Fibonacci cube of order $5$''.

More generally, the \emph{Fibonacci cube of order} $n-1\geq 1$ is canonically obtained as the tope graph of $\mathcal{R}(\{1, \ldots n\},\leq)$, 
where $(\{1, \ldots n\},\leq)$ is the ordered set determined by the (cover) comparabilities $1<3<5<\ldots<2\lfloor\frac{n-1}{2}\rfloor +1$ and 
$2<4<6<\ldots<2\lfloor\frac{n}{2}\rfloor$ and $k<k+3$ for all $k=1, \ldots, n-3$. The incomparable pairs thus form the set $E=\{\{i,i+1\}: i=1, \ldots, n-1\}\}$. 
The intersection graph of $E$ is a path of length $n-1$, which yields a ``fence'' when orienting its edges in a zig-zag fashion. This fence and its 
opposite yield the mutually opposite ordered sets of supremum-irreducibles for the (distributive) ``Fibonacci'' lattice and its opposite. Recall that 
the \emph{opposite} $(R, \leq)^{op}=(R, \leq^{op})$ of an ordered set $(R, \leq)$ is defined by switching $\leq$ to $\geq$, that is: $x\leq^{op} y$ if and only if $y\leq x$.

Similarly, to the fact that hyperplanes, carriers, and open halfspaces of a ranking COM are also ranking COMs, the class of ranking COMs is also closed 
with respect to contractions. On the other hand deleting an element in a ranking COM may give a COM which is not a ranking COM.

To give a small example, consider the minor $\mathcal{R}(P,\leq)\setminus\{5,6\}$ of the lopsided system $\mathcal{R}(P,\leq)$ of Figure~\ref{fig:Fig4}. 
Suppose by way of contradiction, that $\mathcal{R}(P,\leq)\setminus\{5,6\}$ could be represented by some $\mathcal{R}(Q,\leq)$. The ordered set $(Q,\leq)$ 
must be of width $2$ since the tope graph of $\mathcal{R}(Q,\leq)$ is obtained from the graph in Figure~\ref{fig:distlatt} by contracting the five edges 
labeled $56$ and thus includes no $3$-cube. We can keep the current labeling without loss of generality to the point that $Q$ must include four antichains 
$\{1,2\}$, $\{2,3\}$, $\{3,4\}$, $\{3,6\}$ of size $2$ with exactly this intersection pattern. But then the fifth antichain must be disjoint from $\{1,2\}$ 
and $\{2,3\}$ but intersecting both $\{3,4\}$ and $\{3,6\}$, whence it must be $\{4,6\}$, which however yields a contradiction as $\{3,4,6\}$ cannot be an 
antichain in $(Q,\leq)$. Furthermore, $\mathcal{R}(P,\leq)\setminus\{5,6\}$ is easily seen to be the COM amalgam of ranking COMs, i.e., the class is also 
not closed under COM amalgamations.

It would be interesting to determine the smallest minor-closed class of COMs containing the ranking COMs.

\section{COMs as complexes of oriented matroids}\label{sec:complexes}

In this section we consider a topological approach to COMs. In the subsequent definitions, notations, and results we closely follow Section 4 
of \cite{bjvestwhzi-93} (some missing definitions can be also found there). Let $B^d=\{ x\in {\mathbb R}^d: ||x||\le 1\}$ be the {\it standard $d$-ball} 
and its boundary $S^{d-1}=\partial B^d=\{ x\in {\mathbb R}^d: ||x||=1\}$ be the {\it standard $(d-1)$-sphere}. When saying that a topological space $T$ is a 
``ball'' or a ``sphere'', it is meant that $T$ is homeomorphic to  $B^d$ or $S^{d-1}$ for some $d$, respectively.

\subsection{Regular cell complexes}

A {\it (regular) cell complex} $\Delta$ constitutes of a covering of a Hausdorff space $||\Delta||=\bigcup_{\sigma\in \Delta} \sigma$ with finitely many 
subspaces $\sigma$ homeomorphic with (open) balls such that  (i) the interiors of
$\sigma\in \Delta$ partition $||\Delta||$ (i.e., every $x\in ||\Delta||$ lies in the interior of a single $\sigma \in \Delta$), (ii) the boundary $\partial \sigma$ of each
ball $\sigma \in \Delta$ is a union of some members of $\Delta$ \cite[Definition 4.7.4]{bjvestwhzi-93}. Additionally, we will assume that $\Delta$ obeys the {\it intersection property}
(iii) whenever $\sigma,\tau\in \Delta$ have non-empty intersection then $\sigma\cap \tau\in \Delta$. The balls $\sigma\in \Delta$ are called {\it cells} of $\Delta$ and the space $||\Delta||$
is called the {\it underlying space} of $\Delta$. If $T$ is homeomorphic to $||\Delta||$ (notation $T\cong ||\Delta||$), then $\Delta$ is said to provide a 
{\it regular cell decomposition} of the space $T$. We will say that a regular cell complex $\Delta$ is {\it contractible} if the topological space $||\Delta||$ is contractible. 
If $\sigma,\tau\in \Delta$ and $\tau\subseteq \sigma$, then  $\tau$ is said to be a {\it face} of $\sigma$. $\Delta'\subseteq \Delta$ is a {\it subcomplex} 
of $\Delta$ if $\tau\in \Delta'$ implies that every face of $\tau$ also belongs to $\Delta'$.  The 0-cells and 1-cells of $\Delta$ are called {\it vertices} and 
{\it edges}. The {\it 1-skeleton} of $\Delta$ is encoded by the  graph $G(\Delta)$ consisting  of the vertices of $\Delta$ and graph edges corresponding to the 
edges of $\Delta$. The set of cells of $\Delta$ ordered by containment is denoted by ${\mathcal F}(\Delta)$ (in \cite{DaJaSc}, ${\mathcal F}(\Delta)$ is also 
called an {\it abstract cell complex}). Two cell complexes $\Delta$ and $\Delta'$ are {\it combinatorially equivalent} if their ordered sets ${\mathcal F}(\Delta)$ 
and ${\mathcal F}(\Delta')$ are isomorphic. We continue by recalling several results relating regular cell complexes.

The {\it order complex} of a finite ordered set $P$ is an abstract simplicial complex $\Delta_{ord}(P)$  whose vertices are the elements of  $P$ and whose simplices are the
chains $x_0<x_1<\cdots<x_k$ of $P$. The {\it geometric realization} $||\Delta||$ of a complex $\Delta$ basically consists of simultaneously replacing all abstract simplices 
by geometric simplices, see~\cite{BrHa} for a formal definition. In particular, $||\Delta_{ord}(P)||$ is the geometric realization of $\Delta_{ord}(P)$. For an element 
$x$ of $P$ let $P_{<x}=\{ y\in P: y<x\}$ and $P_{\le x}=P_{<x}\cup \{ x\}$. The following fact expresses that a regular cell complex is homeomorphic to the order complex 
of its ordered set of faces.

\begin{proposition}\cite[Proposition 4.7.8]{bjvestwhzi-93} \label{prop:homeo} Let $\Delta$ be a regular cell complex. Then $||\Delta||\cong ||\Delta_{ord}({\mathcal F}(\Delta)||.$ 
Moreover, this homeomorphism can be chosen to be \emph{cellular}, i.e., it restricts to a homeomorphism between $\sigma$ and $||\Delta_{ord}({\mathcal F}_{\le \sigma})||,$ for all 
$\sigma \in \Delta$.
\end{proposition}

The ordered sets of faces of regular cell complexes can be characterized in the following way:

\begin{proposition}\cite[Proposition 4.7.23]{bjvestwhzi-93} \label{prop:faceposet} Let $P$ be an ordered set. Then $P\cong{\mathcal F}(\Delta)$ 
for some regular cell complex $\Delta$ if and only if  $||\Delta_{ord}(P_{<x})||$ is homeomorphic to a sphere for all $x\in P$. Furthermore, $\Delta$ is 
uniquely determined by $P$ up to a cellular homeomorphism.
\end{proposition}

\subsection{Cell complexes of OMs}

Now, let ${\mathcal L}\subseteq \{ \pm1,0\}^E$ be the set of covectors of an oriented matroid. Then  $(\covectors,\le)$ is a semilattice with least element 
$\mathbf{0}$ (where $\le$ is the product ordering on $\{ \pm 1,0\}^E$ defined above).
The semilattice $(\covectors \cup \{ \hat{1}\},\le)$, i.e., the semilattice $\covectors$ with a largest element $\hat{1}$ adjoined, is a lattice, called 
the {\it big face lattice} of $\covectors$ and denoted by ${\mathcal F}_{big}(\covectors)$. Let ${\mathcal F}_{big}(\covectors)^{op}$ denote the opposite 
of  ${\mathcal F}_{big}(\covectors)$.

\begin{proposition}\cite[Corollary 4.3.4 \& Lemma 4.4.1]{bjvestwhzi-93} \label{prop:sphere} Let $(E,\covectors)$ be an oriented matroid of rank $r$. 
Then ${\mathcal F}_{big}(\covectors)^{op}$ is isomorphic to the face lattice of a PL (Piecewise Linear) regular cell decomposition of the $(r-1)$-sphere, 
denoted by $\Delta(\covectors)$. The tope graph of $\covectors$ encodes the 1-skeleton of $\Delta(\covectors)$.
\end{proposition}

\subsection{Cell complexes of COMs}
We collected all ingredients necessary to associate to each COM a regular cell complex.  Let ${\mathcal L}\subseteq \{ \pm1,0\}^E$ be the set of covectors of a COM. 
Analogously to oriented matroids, let ${\mathcal F}_{big}(\covectors):=(\covectors \cup \{ \hat{1}\},\le)$ denote the ordered set $\covectors$ with a top element 
$\hat{1}$ adjoined and call ${\mathcal F}_{big}(\covectors)$ the {\it big  face semilattice} of $\covectors$. Let ${\mathcal F}_{big}(\covectors)^{op}$ denote 
the opposite of ${\mathcal F}_{big}(\covectors)$. ${\mathcal F}_{big}(\covectors)^{op}$ is isomorphic to the semilattice comprising the empty set and the faces 
of $\covectors$ ordered by inclusion.  Recall that for any $X\in \covectors$, the deletion $(E\setminus\underline{X},F(X)\setminus \underline{X})$ corresponding 
to the face $F(X)$ is an oriented matroid, which we will denote  by $\covectors(X)$. Since $F(Y)\subseteq F(X)$ if and only if $Y\in F(X)$, the order ideal 
${\mathcal F}_{big}({\covectors})^{op}_{\le X}$ coincides with the interval  $[\hat{1},X]$ of ${\mathcal F}_{big}({\covectors})^{op}$ and is isomorphic to the 
opposite big face lattice  ${\mathcal F}_{big}({\covectors}(X))^{op}$ of $\covectors(X)$. By Proposition~\ref{prop:sphere}, if $r$ is the rank of $\covectors(X)$, 
then ${\mathcal F}_{big}(\covectors(X))^{op}$ is isomorphic to the face lattice of a PL cell decomposition $\Delta({\covectors}(X))$ of the $(r-1)$-sphere. Additionally, 
the tope graph of $\covectors(X)$ encodes the 1-skeleton of $\Delta({\covectors}(X))$. Denote by $\sigma({\covectors}(X))$ the open PL ball whose boundary is the 
$(r-1)$-sphere occurring in the definition of  $\Delta({\covectors}(X))$. We will call the cells of $\Delta({\covectors}(X))$ {\it faces} of $\sigma({\covectors}(X))$. 
The faces of $\Delta({\covectors}(X))$ correspond to the elements of ${\covectors}(X)\cup\{\hat{1}\}$. Notice in particular that
the adjoined element $\hat{1}$ corresponds to the empty face in $\Delta({\covectors}(X))$ and $\mathbf{0}\in F(X)\setminus \underline{X}$ corresponds to the unique 
maximal face $\sigma({\covectors}(X))$.

By Proposition~\ref{prop:homeo}, for any $X\in \covectors$ we have $||\Delta({\covectors}(X))||\cong ||\Delta_{ord}({\mathcal F}_{big}({\covectors}(X))^{op})||$. 
Furthermore, since ${\mathcal F}_{big}({\covectors}(X))^{op}$ is isomorphic to 
${\mathcal F}_{big}({\covectors})^{op}_{\le X}$, $||\Delta_{ord}({\mathcal F}_{big}({\covectors}(X))^{op})||\cong||\Delta_{ord}({\mathcal F}_{big}({\covectors})^{op}_{\le X})||$. 
Thus for each $X\in \covectors$, $||\Delta_{ord}({\mathcal F}_{big}({\covectors})^{op}_{<X})||$ is homeomorphic to $||\Delta({\covectors}(X))\setminus\sigma(\covectors(X))||$, 
which is a sphere by Proposition~\ref{prop:sphere}. Now, by Proposition~\ref{prop:faceposet},  ${\mathcal F}_{big}(\covectors)^{op}$ is the face semilattice of a regular 
cell complex $\Delta(\covectors)$. Moreover, from the proof of Proposition~\ref{prop:faceposet} it follows that  $\Delta(\covectors)$ can be chosen so that its cells 
are the balls $\sigma({\covectors}(X)), X\in \covectors,$ whose boundary spheres are decomposed by $\Delta({\covectors}(X))$. Since $F(X)\cap F(Y)=F(X\circ Y)$ for 
any two covectors  $X,Y\in \covectors$ such that $F(X)$ and $F(Y)$ intersect, ${\mathcal F}_{big}({\covectors}(X\circ Y))^{op}$ is isomorphic to a sublattice of 
${\mathcal F}_{big}({\covectors}(X))^{op}$ and to a sublattice of
${\mathcal F}_{big}({\covectors}(Y))^{op}$.  Therefore the cells $\Delta({\covectors}(X))$ and $\Delta({\covectors}(Y))$ are glued in $\Delta(\covectors)$ along 
$\Delta({\covectors}(X\circ Y))$, whence $\Delta(\covectors)$ also satisfies the intersection property (iii). Notice also that since the 1-skeleton of each 
$\Delta({\covectors}(X))$ yields the tope graph of $\covectors (X)$ and $\Delta(\covectors)$ satisfies (iii), the 1-skeleton of $\Delta(\covectors)$ encodes 
the tope graph of $\covectors$. We summarize this in the following proposition, in which we also establish that $\Delta(\covectors)$ is contractible.

\begin{proposition} \label{complex} If $(E,\covectors)$ is a COM, then $\Delta(\covectors)$ is a contractible regular cell complex and the tope graph 
of $\covectors$ is realized by the 1-skeleton of $\Delta(\covectors)$.
\end{proposition}

\begin{proof} We prove the contractibility of $\Delta(\covectors)$ by induction on the number of maximal cells of $\Delta(\covectors)$  by using the so-called 
gluing lemma \cite[Lemma 10.3]{Bj} and our decomposition procedure (Proposition~\ref{amalgam}) for COMs. By the gluing lemma, if $\Delta$ is a cell complex which 
is the union of two contractible cell complexes $\Delta'$ and $\Delta''$ such that their intersection $\Delta_0=\Delta'\cap \Delta''$ is contractible, then $\Delta$ 
is contractible.

If $\Delta(\covectors)$ consists of a single maximal cell $\sigma({\covectors}(X))$, then $(E,\covectors)$ is an OM and therefore is contractible. Otherwise, 
as shown above there exists an element $e\in E$ such that if we set $\covectors':=\covectors^-_e$ and $\covectors'':=\overline{\covectors^+_e}$,  then
$(E,\covectors)$ is the COM amalgam of the COMs $(E,\covectors')$ and $(E,\covectors'')$  along the COM $\covectors'\cap \covectors''=N^-(\covectors^0_e)$. 
By induction hypothesis, the cell complexes $\Delta(\covectors'), \Delta(\covectors''),$ and $\Delta(\covectors'\cap \covectors'')$ are contractible. Each maximal 
cell of $\Delta(\covectors)$ corresponds to a maximal face of $\covectors$, thus by Proposition~\ref{amalgam} it is a maximal cell of  $\Delta(\covectors')$,  
of $\Delta(\covectors''),$ or of both (in which case it belongs to  $\Delta(\covectors'\cap\covectors'')$).  Since each cell of  $\Delta(\covectors)$ belongs 
to a maximal cell, this implies that $\Delta(\covectors)\subseteq \Delta(\covectors')\cup \Delta(\covectors'')$. Since $\covectors'\cup \covectors''\subseteq \covectors$, 
we also have the converse inclusion $\Delta(\covectors')\cup \Delta(\covectors'')\subseteq \Delta(\covectors)$.  Finally, since $\covectors'\cap \covectors''=N^-(\covectors^0_e)$, 
the definition of carriers implies that $\Delta(\covectors'\cap \covectors'')=\Delta(\covectors')\cap \Delta(\covectors'')$.
\end{proof}

\subsection{Zonotopal COMs}
As in the introduction, let $E$ be a central arrangement of $n$ hyperplanes of ${\mathbb R}^d$ and $\mathcal L$ be  the oriented matroid corresponding 
to the regions of ${\mathbb R}^d$ defined by this arrangement. Let ${\bf X}=\{ {\bf x}_1,\ldots,{\bf x}_n\}$ be a set of unit vectors each normal to a 
different hyperplane of $E$.  The {\it zonotope} ${\mathcal Z}:={\mathcal Z}({\bf X})$ of ${\bf X}$ is the convex polytope of ${\mathbb R}^d$ which can 
be expressed as the Minkowski sum of $n$ line segments
$${\mathcal Z}=[-{\bf x}_1,{\bf x}_1]+[-{\bf x}_2,{\bf x}_2]+\ldots+[-{\bf x}_n,{\bf x}_n].$$
Equivalently, ${\mathcal Z}$ is the projection of the  $n$-cube $C_n:=\{ \sum_{i=1}^n \lambda_i{\bf e}_i: -1\le \lambda_i\le +1\}\subset {\mathbb R}^n$ 
under $\bf X$ (where ${\bf e}_1,\ldots,{\bf e}_n$ denotes the standard basis of ${\mathbb R}^n$), which sends ${\bf e}_i$ to ${\bf x}_i$, $i=1,\ldots,n$:
$${\mathcal Z}=\{ \sum_{i=1}^n \lambda_i{\bf x}_i: -1\le \lambda_i\le +1\}\subset {\mathbb R}^d.$$
The hyperplane arrangement $E$ is \emph{geometrically polar} to  ${\mathcal Z}$: the regions of the arrangement are the cones of outer normals at the 
faces of  ${\mathcal Z}$. The face lattice of  ${\mathcal Z}$ is opposite (anti-isomorphic) to the big face lattice of the oriented matroid $\mathcal L$ 
of ${\bf X}$, that is, ${\mathcal F}({\mathcal Z})\simeq {\mathcal F}_{big}(\covectors)^{op}$; for this and other results, see \cite[Subsection 2.2]{bjvestwhzi-93}. 
Therefore the zonotopes together with their faces can be viewed as the cell complexes associated to realizable oriented matroids.

The following properties and examples of zonotopes are well-known:
\begin{itemize}
\item any face of a zonotope is a zonotope;
\item a polytope $P$ is a zonotope if and only if every 2-dimensional face of $P$ is a zonotope and if and only if every 2-dimensional face of $P$  is centrally symmetric;
\item two zonotopes are combinatorially equivalent if and only if their 1-skeletons yield isomorphic graphs;
\item the $d$-cube is the zonotope corresponding to the arrangement of coordinate hyperplanes (called also {\it Boolean arrangements} \cite{DaJaSc}) in ${\mathbb R}^d$;
\item the permutohedron is  the zonotope corresponding to the braid arrangement in ${\mathbb R}^d$.
\end{itemize}

A regular cell complex $\Delta$ is a (combinatorial) {\it zonotopal complex} if each cell of $\Delta$ is combinatorially equivalent
to a zonotope \cite{DaJaSc}. Analogously, $\Delta$ is a {\it cube complex} if each of its cells is a combinatorial cube. A {\it geometric zonotopal} or {\it cube} {\it complex}
is a zonotopal (respectively, cube) complex $\Delta$ with a metric such that each face is isometric to a zonotope (respectively, a cube) of the Euclidean space. Moreover,
faces are glued together by isometry along their maximal
common subfaces. The cell complex $\Delta(\covectors)$ associated to a lopsided set $(E,\covectors)$ is a geometric cube complex: $\Delta(\covectors)$ is the union of all
subcubes of the cube $[-1,+1]^E$ whose barycenters are sign vectors from $\covectors$ \cite{bachdrko-12}.

A COM $(E,\covectors)$ is called {\it locally realizable} (or {\it zonotopal}) if  $\covectors(X)$ is a realizable oriented matroid for any $X\in \covectors$. 
Then $\Delta(\covectors)$ is a zonotopal complex because each cell $\Delta(\covectors(X)), X\in \covectors$, is combinatorially equivalent to a zonotope. 
A zonotopal COM $(E,\covectors)$ is called {\it zonotopally realizable } if $\Delta(\covectors)$ is a geometric zonotopal complex. Clearly, zonotopally realizable 
COMs are  locally realizable. The converse is the content of the following question:

\begin{question}\label{quest:1} Is  any  locally realizable COM zonotopally realizable?
\end{question}

\begin{proposition}\label{prop:realtozono} If $\covectors$ is a realizable COM, then $\covectors$ is zonotopally realizable (and thus locally realizable). 
In particular, each ranking COM is zonotopally realizable.
\end{proposition}
\begin{proof}
Since $\covectors$ is realizable there is a set of oriented affine hyperplanes of $\mathbb{R}^d$ and an open convex set $C$, such that $\covectors=\covectors(E,C)$. 
Without loss of generality we can assume that $C$ is the interior of a full-dimensional polyhedron $P$. Let $F$ be the set of supporting hyperplanes of $P$. 
Consider the central hyperplane arrangement $A$ resulting from lifting the affine arrangement $E\cup F$ to $\mathbb{R}^{d+1}$. The associated OM $\covectors'$ 
is realizable and therefore zonotopally realizable. Since $\Delta(\covectors)$ is a subcomplex of $\Delta(\covectors')$,  also $\covectors$ is zonotopally realizable.
\end{proof}

\subsection{CAT(0) Coxeter COMs}

We conclude this section by presenting another class of zonotopally realizable COMs. Namely, we prove that the
CAT(0) Coxeter (zonotopal) complexes introduced in \cite{hapa-98}  arise from COMs. They represent a common
generalization of benzenoid systems (used for illustration in Section~\ref{sec:cocircuit}), 2-dimensional cell
complexes obtained from bipartite cellular graphs \cite{bach-96}, and CAT(0) cube complexes (cube complexes
arising from median structures) \cite{BaCh_survey}. One can say that CAT(0) zonotopal complexes generalize
CAT(0) cube complexes in the same way as COMs generalize lopsided sets.

A zonotope $\mathcal Z$ is called a {\it Coxeter zonotope} (an {\it even polyhedron} \cite{hapa-98}
or a {\it Coxeter cell} \cite{Davis})  if $\mathcal Z$ is symmetric with respect to the mid-hyperplane $H_f$ of each edge $f$ of
$\mathcal Z$, i.e., to the hyperplane perpendicular to $f$ and passing via the middle of $f$.
A cell complex $\Delta$ is called a {\it Coxeter complex} if $\Delta$ is a geometric
zonotopal complex in which each cell is isometric to a Coxeter zonotope. Throughout
this subsection, by $\Delta$ we denote a Coxeter complex and by  $||\Delta||$ the underlying metric
space of $\Delta$.

If $\mathcal Z$ is a Coxeter zonotope and $f,f'$ are two parallel edges of $\mathcal Z$,
then one can easily see that the mid-hyperplanes $H_f$ and $H_{f'}$ coincide. If
${\mathcal Z}=[-{\bf x}_1,{\bf x}_1]+[-{\bf x}_2,{\bf x}_2]+\ldots+[-{\bf x}_n,{\bf x}_n],$
denote by $H_i$ the mid-hyperplane to all edges of $\mathcal Z$ parallel to the segment
$[-{\bf x}_i,{\bf x}_i]$, $i=1,\ldots,n$. Then $\mathcal Z$ is the zonotope of the regions
defined by the arrangement $\{ H_1,\ldots,H_n\}$.
It is well-known \cite[Definition 7.3.1]{Davis} (and is also noticed  in \cite[p.184]{hapa-98}) that Coxeter zonotopes
are exactly the zonotopes
associated to {\it reflection  arrangements} (called also {\it Coxeter arrangements}) of hyperplanes,
i.e., to arrangements of hyperplanes of a finite reflection group \cite[Subsection 2.3]{bjvestwhzi-93}.
For each $i=1,\ldots,n$, denote by ${\mathcal Z}_i$ the intersection of ${\mathcal Z}$ with the
hyperplane $H_i$ and call it a {\it mid-section of $\mathcal Z$}. The mid-sections ${\mathcal Z}_i$
of a Coxeter zonotope $\mathcal Z$ of dimension $d$ are Coxeter zonotopes of dimension~$d-1$.

We continue with the definition of CAT(0) metric spaces and CAT(0) Coxeter complexes.
The underlying space (polyhedron) $||\Delta||$
of a geometric zonotopal  complex  (and, more generally, of a cell
complex with Euclidean convex polytopes as cells) $\Delta$ can be endowed
with an intrinsic $l_2$-metric in the following way.  Assume that
inside every maximal face
of $||\Delta||$ the distance is measured by the $l_2$-metric.
The {\it intrinsic} $l_2$-{\it metric} $d_2$ of $||\Delta||$ is defined by letting
the distance between two points $x,y\in ||\Delta||$ be equal to
the greatest lower bound on the length of the paths joining them;
here a {\it path} in $||\Delta||$ from $x$ to $y$ is a sequence
$x=x_0,x_1,\ldots, x_m=y$ of points in $||\Delta||$ such that
for each $i=0,\ldots, m-1$ there exists a face $\sigma_i$ containing
$x_i$ and $x_{i+1},$ and the {\it length} of the path equals
$\sum_{i=0}^{m-1} d(x_i,x_{i+1}),$ where $d(x_i,x_{i+1})$ is
computed inside $\sigma_i$ according to the respective $l_2$-metric.  The
resulting metric space is {\it geodesic}, i.e., every pair of points
in $||\Delta||$ can be joined by a geodesic; see \cite{BrHa}.

A {\it geodesic triangle} $T:=T(x_1,x_2,x_3)$ in a geodesic
metric space $(X,d)$ consists of three points in $X$ (the vertices
of $T$) and a geodesic  between each pair of vertices (the
edges of $T$). A {\it comparison triangle} for $T(x_1,x_2,x_3)$
is a triangle $T(x'_1,x'_2,x'_3)$ in the
Euclidean plane  ${\mathbb R}^2$ such that $d_{{\mathbb
R}^2}(x'_i,x'_j)=d(x_i,x_j)$  for $i,j\in \{ 1,2,3\}.$ A geodesic
metric space $(X,d)$ is a {\it CAT(0) space}
\cite{Gr} if all geodesic triangles $T(x_1,x_2,x_3)$ of $X$
satisfy the comparison axiom of Cartan--Alexandrov--Toponogov:
{\it If $y$ is a point on the side of $T(x_1,x_2,x_3)$ with
vertices $x_1$ and $x_2$ and $y'$ is the unique point on the line
segment $[x'_1,x'_2]$ of the comparison triangle
$T(x'_1,x'_2,x'_3)$ such that $d_{{\mathbb R}^2}(x'_i,y')=
d(x_i,y)$ for $i=1,2,$ then $d(x_3,y)\le d_{{\mathbb
R}^2}(x'_3,y').$}

CAT(0) spaces can be characterized  in several different  natural ways
and have numerous  properties (for a full account of this theory consult
the book \cite{BrHa}). For instance, a cell complex endowed with a piecewise
Euclidean metric is CAT(0) if and only if any two points can be joined by a
unique geodesic. Moreover, CAT(0) spaces are
contractible.

A {\it CAT(0) Coxeter complex} is a Coxeter complex $\Delta$ for which $||\Delta||$
endowed with the intrinsic $l_2$-metric $d_2$ is a CAT(0) space. In this case, the
parallelism relation on edges of cells of $\Delta$ induces
a parallelism relation on all edges of $\Delta$:  two edges $f,f'$ of $\Delta$ are {\it parallel} if
there exists a sequence of edges
$f_0=f,f_1,\ldots, f_{k-1},f_k=f'$ such that any two consecutive edges $f_{i-1},f_i$ are parallel
edges of a common cell of $\Delta$. Parallelism is an equivalence relation on the edges of $\Delta$.
Denote by $E$ the equivalence classes of this parallelism relation. 
For $e\in E$, we denote by $\Delta_e$  the union of all mid-sections of the form  ${\mathcal Z}_e$ for cells
$\mathcal Z\in \Delta$ which contain edges from the equivalence class $e$ (let $||\Delta_e||$ be
the underlying space of $\Delta_e$).  We call
each $||\Delta_e||$ (or $\Delta_e$), $e\in E$, a {\it mid-hyperplane} (or a {\it wall} as in \cite{hapa-98}) of $||\Delta||$.
Since each mid-section  included in  $\Delta_e$ is a Coxeter zonotope, each mid-hyperplane of a Coxeter complex is a
Coxeter complex as well. CAT(0) Coxeter complexes have additional nice and strong properties, which  have
been established in \cite{hapa-98}.

\begin{lemma} \cite[Lemme 4.4]{hapa-98} \label{convex} Let $\Delta$ be a CAT(0) Coxeter complex and $\Delta_e$
be a mid-hyperplane of $\Delta$. Then $||\Delta_e||$ is a convex subset of $||\Delta||$  
and $||\Delta_e||$ partitions $||\Delta||$
in two connected components  $||\Delta_e^-||$ and $||\Delta_e^+||$ (called halfspaces of $||\Delta||$).
\end{lemma}

If $x\in ||\Delta^-_e||$ and $y\in ||\Delta^+_e||$, then $x$ and $y$ are said to be {\it separated} by the mid-hyperplane (wall) $||\Delta_e||$.
A path $P$ in $||\Delta||$ {\it traverses} a mid-hyperplane $||\Delta_e||$ if $P$ contains an edge $xy$ such that $x$ and $y$ are separated by $||\Delta_e||$.
Two distinct mid-hyperplanes $||\Delta_e||$ and $||\Delta_f||$ are called {\it parallel} if $||\Delta_e||\cap ||\Delta_f||=\varnothing$ and {\it crossing} if
$||\Delta_e||\cap ||\Delta_f||\ne\varnothing$.

\begin{lemma} \cite[Corollaire 4.10]{hapa-98} \label{onewall} Two vertices $u,v$ of $\Delta$ are adjacent in $G(\Delta)$ if and only if $u$ and $v$ are
separated by a single mid-hyperplane of $||\Delta||$.
\end{lemma}

\begin{lemma} \cite[Proposition 4.11]{hapa-98} \label{isometric_embedding} A path $P$ of $G(\Delta)$ between two vertices $u,v$  is a shortest
$(u,v)$-path in $G(\Delta)$ if and only if $P$ traverses each mid-hyperplane of $||\Delta||$ at most once. 
\end{lemma}

These three results imply that the arrangement of mid-hyperplanes of a CAT(0) Coxeter complex $\Delta$ defines a wall system sensu \cite{hapa-98}, which in turn
provides us with a system $\covectors (\Delta)$ of sign vectors.  Define the mapping $\varphi: \Delta\rightarrow \{ \pm 1,0\}^E$ in the following way.
First, for $e\in E$ and $x\in \Delta$, set
$$
\varphi_e(x):=\begin{cases}
-1 &\mbox{ if } x\in ||\Delta^-_e||,\\
0  &\mbox{ if } x\in ||\Delta_e||,\\
+1 &\mbox{ if } x\in ||\Delta^+_e||.
\end{cases}
$$
Let $\varphi(x)=(\varphi_e(x): e\in E)$. Denote by $\covectors(\Delta)$ the set of all sign vectors of the form $\varphi(x), x\in ||\Delta||$.
Notice that if a point $x$ of $||\Delta||$ does
not belong to any mid-hyperplane of $||\Delta||$, then $\varphi(x)\in \{ \pm 1\}^E$; in particular, this is the case for the vertices of $G(\Delta)$.
Moreover, Lemma~\ref{isometric_embedding} implies that $\varphi$ defines
an isometric embedding of $G(\Delta)$ into the hypercube $\{\pm 1\}^E$.

\begin{theorem} \label{zonotopalCOM} Let $\Delta$ be a CAT(0) Coxeter complex, $E$ be the classes of parallel edges of $\Delta$,
and $\covectors(\Delta):=\cup \{ \varphi(x): x\in ||\Delta||\}\subseteq \{ \pm 1,0\}^E$. Then  $(E, \covectors(\Delta))$
is a simple COM and $G(\Delta)$ is its tope graph.
\end{theorem}

\begin{proof} We proceed by induction on the size of $E$. It suffices to show that  $(E,\covectors(\Delta))$ is simple
and satisfies the conditions (1),(3),(4), and (2$'$) of Theorem~\ref{hyperplane}.
That $\covectors(\Delta)$ satisfies (N1) is evident. To verify the condition (N2),
let $e,f\in E$. If the mid-hyperplanes $||\Delta_e||$ and $||\Delta_f||$
cross, then the four intersections $||\Delta^-_e||\cap ||\Delta^-_f||,||\Delta^-_e||\cap ||\Delta^+_e||, ||\Delta^+_e||\cap ||\Delta^-_f||,$
and $||\Delta^+_e||\cap ||\Delta^+_f||$ are nonempty, and  as
$X$ and $Y$ with $\{ X_eX_f,Y_eY_f\}=\{ \pm 1\}$ one can pick the sign vectors $\varphi(x)$ and $\varphi(y)$ of any
two points $x\in ||\Delta^+_e||\cap ||\Delta^-_f||$ and $y\in ||\Delta^-_e||\cap ||\Delta^-_f||$.  Otherwise, if $||\Delta_e||$
and $||\Delta_f||$ are parallel, then one  of the four pairwise intersections of halfspaces is empty, say
$||\Delta^-_e||\cap ||\Delta^+_f||=\varnothing$, and as $X$ and $Y$ one can take the sign vectors $\varphi(x)$ and $\varphi(y)$
of any points $x\in ||\Delta^+_e||\cap ||\Delta^+_f||$ and $y\in ||\Delta^+_e||\cap ||\Delta^-_f||$.
This establishes that $(E, \covectors(\Delta))$ is simple.

Notice that the tope graph of $(E, \covectors(\Delta))$ coincides with $G(\Delta)$. Indeed, let $X$ be a
tope of $\covectors(\Delta)$. Then  $X=\varphi(x)$ for some $x\in ||\Delta||$.
Let $x\in {\mathcal Z}$ for a cell $\mathcal Z$ of $||\Delta||$. The sign maps of $\covectors(\Delta)$
restricted to $\mathcal Z$ define an oriented matroid whose topes are the vertices
of $\mathcal Z$. Therefore $\mathcal Z$ contains a vertex $v$ such that $\varphi(v)=\varphi(x)=X$,
whence each tope of $\covectors (\Delta)$ is a vertex of $G(\Delta)$. Conversely,
since each vertex $v$ of $G(\Delta)$ does not belong to any
mid-hyperplane of $||\Delta||$, $\varphi(v)$ is a vertex of $\{ \pm 1\}^E$, and thus a
tope of $\covectors(\Delta)$. This shows that the tope graph of  $\covectors(\Delta)$ and
the 1-skeleton of $\Delta$ have the same sets of vertices. Lemma~\ref{onewall} implies
that two vertices $u$ and $v$ are adjacent in $G(\Delta)$ if and only if they are adjacent
in the tope graph of $\covectors (\Delta)$.
Then Lemma~\ref{isometric_embedding} establishes the condition (3). The condition (4)
immediately follows from the definition of $\covectors(\Delta)$.

Now we establish condition (2$'$) that all hyperplanes $\covectors^0_e(\Delta)$ of $(E,\covectors(\Delta))$ are COMs.
Notice that $X\in \covectors^0_e(\Delta)$ if and only if $X=\varphi(x)$ for some $x\in ||\Delta_e||$. Therefore the hyperplane
$\covectors^0_e(\Delta)$ of $\covectors(\Delta)$ coincides with the restriction of $\covectors(\Delta)$ to the points of
the mid-hyperplane $||\Delta_e||$. Hence $\covectors^0_e(\Delta)$ can be viewed as $\covectors(\Delta_e)$, where
$\covectors(\Delta_e)$ is the set of all sign vectors of $\{ \pm 1,0\}^E$ of the form $\varphi(x), x\in ||\Delta_e||$.
By Lemma~\ref{convex}, $||\Delta_e||$ is a convex subset of $||\Delta||$, thus $\Delta_e$ is a CAT(0) Coxeter complex.
Let $E'$ denote the classes  of  parallel edges of $\Delta_e$; namely,  $E'$ consists of all $f\in E\setminus \{ e\}$ such that
the mid-hyperplanes $||\Delta_e||$ and $||\Delta_{f}||$ are crossing. Notice that  the $f$th mid-hyperplane
$||(\Delta_e)_{f}||$ of $||\Delta_e||$ is just the intersection $||\Delta_e||\cap ||\Delta_f||$.
Analogously, the halfspaces $||(\Delta_e)^-_f||$ and $||(\Delta_e)^+_f||$ coincide with the intersections
$||\Delta^-_f||\cap ||\Delta_e||$ and $||\Delta^+_f||\cap ||\Delta_e||$, respectively. Define
$\varphi': ||\Delta_e||\rightarrow \{ \pm 1,0\}^{E'}$ as follows. For $x\in ||\Delta_e||$ and $f\in E'$, set
$$
\varphi'_f(x):=\begin{cases}
-1 &\mbox{ if } x\in ||(\Delta_e)^-_f||,\\
0  &\mbox{ if } x\in ||(\Delta_e)_f||,\\
+1 &\mbox{ if } x\in ||(\Delta_e)^+_f||.
\end{cases}
$$
Let $\varphi'(x)=(\varphi'_f(x): f\in E')$. Denote by $\covectors'(\Delta_e)$ the set of all sign vectors of the
form $\varphi'(x), x\in ||\Delta_e||$. By the induction hypothesis, $(E', \covectors'(\Delta_e))$ is a COM.
For any point $x\in ||\Delta_e||$, $\varphi(x)$ coincides with $\varphi'(x)$ on $E'$.
Since  $\varphi_e(x)=0$, $\varphi_{e'}(x)=-1$ if $||\Delta_{e'}||$ is parallel to $||\Delta_e||$
and $||\Delta_e||\subset ||\Delta^-_{e'}||,$ and $\varphi_{e'}(x)=+1$ if $||\Delta_{e'}||$ is
parallel to $||\Delta_e||$ and $||\Delta_e||\subset ||\Delta^+_{e'}||,$ $\covectors(\Delta_e)$
can be obtained from $\covectors'(\Delta_e)$ by adding to all sign vectors of $\covectors'(\Delta_e)$
in each coordinate of $E\setminus E'$ a respective constant $0,-1,$ or $+1$. Hence
$(E,\covectors(\Delta_e))$ is a COM, thus establishing (2$'$).

Finally, we show that $\covectors(\Delta)$ satisfies the condition (1), i.e., the composition rule (C). Let $X$ and $Y$
be two sign vectors of $\covectors(\Delta)$ and  $x$ and $y$ be two points of $||\Delta||$ such that $\varphi(x)=X$ and
$\varphi(y)=Y$. As in the case of realizable COMs  presented in the introduction,  connect the two points
$x$ and $y$ by the unique geodesic $\gamma(x,y)$ of $||\Delta||$ and choose $\epsilon> 0$ small enough so
that the open ball of radius $\epsilon$ around $x$ intersects only those mid-hyperplanes of $||\Delta||$ on which $x$ lies.
Pick any point $w$ from the intersection of this $\epsilon$-ball with $\gamma(x,y)\setminus \{ x\}$ and let $W=\varphi(w)$.
We assert that $W=X\circ Y$. Pick any $e\in E$. First suppose that $X_e\ne 0.$ From the choice of $w$ it
immediately follows that $W_e=\varphi_e(w)=\varphi_e(x)=X_e.$ Now suppose that $X_e=0$. If $Y_e=0,$ then
$x,y\in ||\Delta_e||$. Since by Lemma~\ref{convex}, $||\Delta_e||$ is a convex subset of $||\Delta||$,
we have $w\in \gamma(x,y)\subset ||\Delta_e||$, whence $W_e=0=X_e\circ Y_e$. Finally, if $Y_e\ne 0$,
say $Y_e=+1$, then since the set $||\Delta^+_e||\cup ||\Delta_e||$ is convex, either $w\in ||\Delta^+_e||$
or $w\in ||\Delta_e||$. In the first case, we have $W_e=+1=X_e\circ Y_e$ and we are done.
On the other hand, we will show below that the case $w\in ||\Delta_e||$ is impossible.

\begin{figure}[htb]
\includegraphics[width = .7\textwidth]{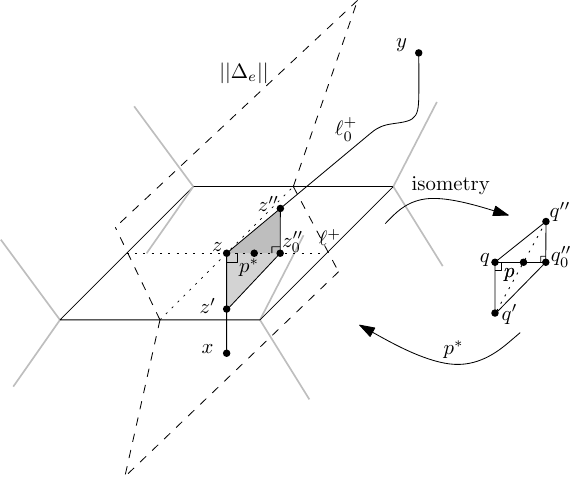}
\caption{To the proof of composition rule in Theorem~\ref{zonotopalCOM}.}
\label{fig:zonotopal}
\end{figure}

Indeed, suppose by way of contradiction that $w\in ||\Delta_e||$. Since $x\in ||\Delta_e||,y\in ||\Delta^+_e||,$
and $||\Delta_e||$ is convex, there exists a point $z\in \gamma(x,y)$ such that $\gamma(x,z)\subset ||\Delta_e||$
and $\gamma(z,y)\setminus\{ z\}\subset ||\Delta^+_e||$. Pick points $z'\in \gamma(x,z)$ and $z''\in \gamma(z,y)$
close enough to $z$ such that each of the couples $z',z$ and $z,z''$ belongs to a common cell of $||\Delta||$.
The choice of $z$ implies that $z',z,z''$ cannot all belong to a common  cell of $||\Delta||$. Denote by
${\mathcal Z}'$ a maximal cell containing $z',z$ and by ${\mathcal Z}''$ a maximal cell containing $z,z''$.
Then $z$ belongs to a common face ${\mathcal Z}_0$ of ${\mathcal Z}'$ and ${\mathcal Z}''$.  Let $\Pi'$ and
$\Pi''$ denote the supporting Euclidean  spaces of  ${\mathcal Z}'$ and ${\mathcal Z}''$, respectively. Let
$\ell$ be the line in $\Pi''$ passing via the point $z$ and parallel to the edges of ${\mathcal Z}''$ from
the equivalence class $e$ and let $\ell^+$ be the ray with origin $z$ and containing the points of $\ell\cap ||\Delta^+_e||$
(this intersection is a non-empty half-open interval). Let $\ell_0$ be the line in $\Pi''$  passing via $z$ and
$z''$ and let $\ell^+_0$ be the ray of $\ell_0$ with origin $z$ and containing $[z,z'']$.  Since $\ell$ is
orthogonal to the supporting plane of the mid-section ${\mathcal Z}'_0={\mathcal Z}'\cap \Delta_e$, the angle
between $\ell^+_0$ and $[z,z']$ is $\frac{\pi}{2}$. Now suppose that $z''$ is selected so close to $z$ that
the orthogonal projection $z''_0$ of $z''$ on the line $\ell$ belongs to the intersection ${\mathcal Z}_0\cap {\ell}^+$.

As a result, we obtain two right-angled triangles $T(z,z',z''_0)$ and $T(z,z'',z''_0)$, the first belonging to $\Pi'$ and the 
second belonging to $\Pi''$ (see Figure~\ref{fig:zonotopal} for an illustration). Therefore, their union is isometric to a convex 
quadrilateral $Q=Q(q,q',q''_0,q'')$ in ${\mathbb R}^2$ having the sides $qq', q'q''_0,q''_0q'',$ and $q''q$ of lengths 
$d_2(z,z'),d_2(z',z''_0),d_2(z''_0,z''),$ and $d_2(z'',z)$, respectively. Let $p$ be the intersection of the diagonals 
$q'q''$ and $qq''_0$ of $Q$ and let $p^*$ be the point (of $\Delta$) on the segment $[z,z''_0]$ such that 
$d_2(z,p^*)=d_{{\mathbb R}^2}(q,p)$ and $d_2(z''_0,p^*)=d_{{\mathbb R}^2}(q''_0,p)$. Then
\begin{align*}
d_2(z',z)+d_2(z,z'')=d_{{\mathbb R}^2}(q',q)+d_{{\mathbb R}^2}(q,q'')>d_{{\mathbb R}^2}(q',q'')\\
=d_{{\mathbb R}^2}(q',p)+d_{{\mathbb R}^2}(p,q'')=d_2(z',p^*)+d_2(p^*,z''),\hspace*{1.4cm}
\end{align*}
contrary to the assumption that $z\in \gamma(z',z'')\subset \gamma(x,y)$.  This establishes our claim and concludes the 
proof that $\covectors(\Delta)$ satisfies the composition rule (C).
\end{proof}

We conclude this section by showing that in fact all COMs having square-free tope graphs arise from 2-dimensional zonotopal COMs:

\begin{proposition}
A square-free partial cube $G$ is the tope graph of a COM ${\mathcal L}(G)$ if and only if $G$ does
not contain an 8-cycle with two subdivided diagonal chords
(graph $X$ in Fig.1 in \cite{Ti}) as an isometric subgraph. The resulting ${\mathcal L}(G)$ is a zonotopal COM.
\end{proposition}

\begin{proof}
Notice that a square-free tope graph $G$ of a COM does not contain  $X$ as an isometric
subgraph. Indeed, since $G$ is square-free, the four 6-cycles of $X$ are convex and, moreover,
$X$ must be a convex subgraph of $G$. Since $X$ isometrically embeds into a 4-cube, it can be
directly checked that $X$ is not the tope graph as a COM, consequently, $X$ cannot occur in
the tope graph of a COM.

Conversely, let $G$ be a square-free partial cube not containing $X$ as an isometric subgraph.
By Proposition 2.6. of \cite{Ti}, any pair of isometric cycles intersect in at most
one edge. By replacing  each isometric cycle of $G$ with a regular polygon with the same number of edges,
we get a $2$-dimensional Coxeter zonotopal complex $||\Delta(G)||$.
Since isometric cycles of a graph generate the cycle space, this complex is simply connected.
Since the sum of angles around any vertex of $||\Delta(G)||$ is at least $2\pi$,
by Gromov's result for $2$-dimensional complexes \cite[p.215]{BrHa},  $||\Delta(G)||$ is CAT(0).
Thus, $\Delta(G)$ is a 2-dimensional CAT(0) Coxeter zonotopal complex and
by Theorem \ref{zonotopalCOM}, $G$ can be realized as the tope graph of a COM.
\end{proof}

\section{Concluding remarks}\label{sec:conclude}

In this paper, we show how COMs naturally arise as a generalization of
oriented matroids and lopsided sets by relaxing the covector axioms. Furthermore,
we give several descriptions of COMs, in particular, in terms of cocircuit axioms.
Nevertheless, such important features of the theory of OMs like duality and topological
representation still lack generalization. We believe that the following problem,
which is well-known for OMs and lopsided sets, is an important next step.

\begin{prob}
Establish duality theory for COMs.
\end{prob}

By Proposition~\ref{carriers}, the halfspaces of a COM are  COMs. Particular examples are the affine oriented
matroids, which are halfspaces of OMs. Even stronger, Lemma~\ref{lem:fiber} shows that the intersections of halfspaces, i.e., the fibers,
of a COM are COMs. While the following is true by definition for realizable COMs (see for instance the proof of Proposition~\ref{prop:realtozono}),
we believe that every COM arises this way from an OM:

\begin{conjecture}\label{conj:1}
 Every COM is a fiber of some OM.
\end{conjecture}

This generalizes the corresponding conjecture of Lawrence~\cite{la-83} that lopsided sets are fibers of uniform OMs. 
Not only would Conjecture~\ref{conj:1} be a natural generalization
of the realizable situation, but using the Topological Representation Theorem of Oriented Matroids \cite{fo-la-78} it will also give
a natural topological representation for COMs. In fact, Conjecture~\ref{conj:1} is equivalent to the following conjecture: For every
COM $\mathcal L$ there is a number $d$ such that $\mathcal L$ can be represented by a set of $(d-2)$-dimensional
pseudospheres restricted to the intersection of a set of open $(d-1)$-dimensional pseudoballs inside a $(d-1)$-sphere.

For locally realizable COMs, the following version of Conjecture~\ref{conj:1} would imply a positive answer to Question~\ref{quest:1}:

\begin{conjecture}\label{conj:2}
Every locally realizable COM is a fiber of a realizable OM.
\end{conjecture}

Conjecture~\ref{conj:2} can be rephrased as: The tope graph of a locally realizable COM is a convex subgraph of the 1-skeleton of a zonotope.
Analogously, Conjecture~\ref{conj:1} can be rephrased as: The tope graph of a COM is a convex subgraph of the tope graph of an OM.

\bigskip\noindent
{\bf Acknowledgements.} We would like to acknowledge the referees for useful remarks improving the presentation. Furthermore, we thank Matja\v{z} 
Kov\v{s}e (Leipzig) for initial discussions about partial cubes and tope graphs of oriented matroids and are grateful to Andrea Baum, Tabea Beese, 
and especially Yida Zhu (Hamburg) for careful reading of previous versions and suggesting a simpler argument in the proof of Theorem~\ref{thm:cc_ses}. 
V.C. and K.K. were partially supported by ANR project TEOMATRO ({\sc ANR-10-BLAN-0207}). K.K was furthermore supported by PEPS grant EROS and ANR project GATO ({\sc ANR-16-CE40-0009-01}).

\bibliographystyle{my-siam}

\bibliography{OMC}

\end{document}